\documentclass[12pt,a4paper]{amsart}

\usepackage{fullpage} 

\usepackage{amssymb}
\usepackage{amsrefs}
\usepackage[all]{xy}

\usepackage{color} 

\SelectTips{cm}{}

\numberwithin{equation}{section}
\newtheorem{theo}[equation]{Theorem}
\newtheorem{coro}[equation]{Corollary}
\newtheorem{lemm}[equation]{Lemma}
\newtheorem{prop}[equation]{Proposition}

\newcommand{\bh}{\mathbf{h}}
\newcommand{\bD}{\mathbf{D}}
\newcommand{\bK}{\mathbf{K}}

\newcommand{\bbL}{\mathbb{L}}
\newcommand{\bbN}{\mathbb{N}}
\newcommand{\bbZ}{\mathbb{Z}}

\newcommand{\calA}{\mathcal{A}}
\newcommand{\calC}{\mathcal{C}}
\newcommand{\calF}{\mathcal{F}}
\newcommand{\calG}{\mathcal{G}}
\newcommand{\calI}{\mathcal{I}}
\newcommand{\calP}{\mathcal{P}}
\newcommand{\calR}{\mathcal{R}}
\newcommand{\calT}{\mathcal{T}}

\let\mod=\undefined
\DeclareMathOperator{\Id}{Id}
\DeclareMathOperator{\tr}{tr}
\DeclareMathOperator{\Aut}{Aut}
\DeclareMathOperator{\Hom}{Hom}
\DeclareMathOperator{\ind}{ind}
\DeclareMathOperator{\Irr}{Irr}
\DeclareMathOperator{\mod}{mod}
\DeclareMathOperator{\rad}{rad}
\DeclareMathOperator{\proj}{proj}

\newcommand{\ol}{\overline}

\newcommand{\vertexD}[1]{\bullet \save*+!D{\scriptstyle #1} \restore}
\newcommand{\vertexL}[1]{\bullet \save*+!L{\scriptstyle #1} \restore}

\newcommand{\vertexU}[1]{\bullet \save*+!U{\scriptstyle #1} \restore}

\makeatletter

\newcommand*{\da@rightarrow}{\mathchar"0\hexnumber@\symAMSa 4B }
\newcommand*{\da@leftarrow}{\mathchar"0\hexnumber@\symAMSa 4C }
\newcommand*{\xdashrightarrow}[2][]{%
  \mathrel{%
    \mathpalette{\da@xarrow{#1}{#2}{}\da@rightarrow{\,}{}}{}%
  }%
}
\newcommand{\xdashleftarrow}[2][]{%
  \mathrel{%
    \mathpalette{\da@xarrow{#1}{#2}\da@leftarrow{}{}{\,}}{}%
  }%
}
\newcommand*{\da@xarrow}[7]{%
  \sbox0{$\ifx#7\scriptstyle\scriptscriptstyle\else\scriptstyle\fi#5#1#6\m@th$}%
  \sbox2{$\ifx#7\scriptstyle\scriptscriptstyle\else\scriptstyle\fi#5#2#6\m@th$}%
  \sbox4{$#7\dabar@\m@th$}%
  \dimen@=\wd0 %
  \ifdim\wd2 >\dimen@
    \dimen@=\wd2 %
  \fi
  \count@=2 %
  \def\da@bars{\dabar@\dabar@}%
  \@whiledim\count@\wd4<\dimen@\do{%
    \advance\count@\@ne
    \expandafter\def\expandafter\da@bars\expandafter{%
      \da@bars
      \dabar@
    }%
  }%
  \mathrel{#3}%
  \mathrel{%
    \mathop{\da@bars}\limits
    \ifx\\#1\\%
    \else
      _{\copy0}%
    \fi
    \ifx\\#2\\%
    \else
      ^{\copy2}%
    \fi
  }%
  \mathrel{#4}%
}

\makeatother

\title[Derived equivalences for the derived discrete algebras]{Derived equivalences for the derived discrete algebras are standard}

\author{Grzegorz Bobi\'nski}

\address{Grzegorz Bobi\'nski\newline
Faculty of Mathematics and Computer Science\newline
Nicolaus Copernicus University\newline
ul. Chopina 12/18\newline
87-100 Toru\'n\newline
Poland}
\email{gregbob@mat.umk.pl}

\author{Tomasz Ciborski}

\address{Tomasz Ciborski\newline
Faculty of Mathematics and Computer Science\newline
Nicolaus Copernicus University\newline
ul. Chopina 12/18\newline
87-100 Toru\'n\newline
Poland}
\email{tomaszcib@mat.umk.pl}



\begin{document}

\begin{abstract}
We prove that any derived equivalence between derived discrete algebras is standard, i.e.\ is isomorphic to the derived tensor product by a two-sided tilting complex.
\end{abstract}

\maketitle

\renewcommand{\thefootnote}{}

\footnote{\emph{Corresponding Author}. Grzegorz Bobi\'nski}

\section{Introduction and the main result} \label{sect:intro}

Throughout the paper $\Bbbk$ denotes an algebraically closed field of arbitrary characteristic. All considered categories and functors are assumed to be $\Bbbk$-linear.

For a finite dimensional $\Bbbk$-algebra $A$ we denote by $\mod A$ the category of finite dimensional left $A$-modules and by $\bD^b (\mod A)$ the bounded derived category of $\mod A$. Algebras $A$ and $B$ are said to be \emph{derived equivalent} if there exists a triangle equivalence between $\bD^b (\mod A)$ and $\bD^b (\mod B)$ (such equivalences are called \emph{derived equivalences} between $A$ and $B$). Rickard has proved~\cite{Rickard}*{Theorem~3.3} that if algebras $A$ and $B$ are derived equivalent, then there exists a complex $X$ of $B$-$A$-bimodules such that the derived tensor product $X \otimes^\bbL_A - \colon \bD^b (\mod A) \to \bD^b (\mod B)$ is a derived equivalence. A derived equivalence is called \emph{standard} if it is isomorphic to a derived equivalence of the form $X \otimes^\bbL_A -$, for a complex $X$ of $B$-$A$-bimodules. An open question posed by Rickard is whether every derived equivalence between finite dimensional $\Bbbk$-algebras is standard. This question has been answered affirmatively for some classes of algebras, including the triangular and (anti-)Fano algebras~\cites{Chen, Minamoto, MiyachiYekutieli}. Moreover, Orlov's theorem~\cite{Orlov}*{Theorem~2.2} stating that every derived equivalence between smooth projective varieties is a Fourier--Mukai transform can be viewed as a geometric analogue of these results.

Derived discrete algebras are algebras introduced by Vossieck~\cite{Vossieck} with nontrivial derived categories which are still accessible for direct calculations. This class of algebras is closed under derived equivalences and has been an object of intensive studies. It serves as a test case for verifying and studying homological conjectures and problems (see for example~\cites{ArnesenLakingPauksztelloPrest, Bobinski2011b, BobinskiKrause, BobinskiSchmude, Broomhead, BroomheadPauksztelloPloog2016, BroomheadPauksztelloPloog2017, Qin}). In particular, Chen and Zhang~\cite{ChenZhang}*{Theorem~3.6} have proved that the derived equivalences between derived discrete algebras of finite global dimension are standard. The following main result of the paper extends this result to arbitrary derived discrete algebras.

\begin{theo} \label{theo:main_general}
If $A$ and $B$ are derived discrete $\Bbbk$-algebras, then every derived equivalence between $A$ and $B$ is standard.
\end{theo}

Recall that if $A$ and $B$ are derived equivalent, then $A$ is of finite global dimension if and only if $B$ is of finite global dimension (see \cite{Happel}*{Lemma~III.1.5}). Consequently, in view of the above-mentioned result of Chen and Zhang (\cite{ChenZhang}*{Theorem~3.6}) Theorem~\ref{theo:main_general} follows from the following.

\begin{theo} \label{theo:main_special}
If $A$ and $B$ are derived discrete $\Bbbk$-algebras of infinite global dimension, then every derived equivalence between $A$ and $B$ is standard.
\end{theo}

Theorem~\ref{theo:main_special} has already been proved in some cases by Chen and Ye (see~\cite{ChenYe}*{Section~7}). In this paper we present a unified proof, which also covers the cases studied by Chen and Ye. A crucial fact we use in the proof is Proposition~\ref{prop:auto}, which allows us to use a criterion of Chen and Ye~\cite{ChenYe} (see Proposition~\ref{prop:standard}).

The paper is organized as follows. In Section~\ref{sect:prelim} we present a method developed by Chen and Ye in~\cite{ChenYe} which can be used to verify that, for a given algebra $A$, all derived equivalences starting at $\bD^b (\mod A)$ are standard. We also give a list of representatives of the derived equivalence classes of the derived discrete algebras of infinite global dimension there. Next, Sections~\ref{sect:objmorph} and~\ref{sect:category} are devoted to a description of the homotopy categories of projective modules for the derived discrete algebras of infinite global dimension. Finally in Section~\ref{sect:proof} we prove Theorem~\ref{theo:main_special}.

\section{Preliminaries} \label{sect:prelim}

Throughout the paper by $\bbZ$, $\bbN$, and $\bbN_+$ we denote the sets of integers, nonnegative integers, and positive integers, respectively. If $i, j \in \bbZ$, then $[i, j]$ denotes the set of all $k \in \bbZ$ such that $i \leq k \leq j$. Similarly, if $i \in \bbZ$, then by $[i, \infty)$ ($(-\infty, i]$) we denote the set of all $k \in \bbZ$ such that $i \leq k$ ($k \leq i$, respectively).

We introduce now basic definitions following the exposition in~\cite{ChenYe}. Let $\calT$ and $\calT'$ be triangulated categories with respective suspension functors $\Sigma$ and $\Sigma'$. By a \emph{triangle functor} from $\calT$ to $\calT'$ we mean a pair $(F, \omega)$, where $F \colon \calT \to \calT'$ is a functor and $\omega \colon F \Sigma \to \Sigma' F$ is a natural isomorphism, such that the triangle $F X \xrightarrow{F f} F Y \xrightarrow{F g } F Z \xrightarrow{\omega_X \circ F h} \Sigma' F X$ is exact in $\calT'$, for every triangle $X \xrightarrow{f} Y \xrightarrow{g} Z\xrightarrow {h} \Sigma X$ which is exact in $\calT$. In the above situation $\omega$ is called a \emph{connecting isomorphism}. A \emph{natural transformation} $\eta \colon (F, \omega) \to (F', \omega')$ between two triangle functors $(F, \omega), (F', \omega') \colon \calT \to \calT'$ is a natural transformation from $F$ to $F'$ such that $\omega_X' \circ \eta_{\Sigma X} = \Sigma' \eta_X \circ \omega_X$, for every object $X$ in $\calT$.

Let $\calA$ be a category. We denote by $\bK^b (\calA)$ the bounded homotopy category of $\calA$. We can identify $\calA$ with the full subcategory of $\bK^b (\calA)$ formed by the stalk complexes concentrated in degree zero. We call a triangle functor $(F, \omega) \colon \bK^b (\calA) \to \bK^b (\calA)$ a \emph{pseudo-identity} if $F (X) = X$, for every object $X$ of $\bK^b (\calA)$, and $F |_{\Sigma^n \calA}$ coincides with the identity functor $\Id_{\Sigma^n \calA}$, for all $n \in \bbZ$. It follows from~\cite{ChenYe}*{Corollary~3.4} that any pseudo-identity is necessarily an autoequivalence.

For a finite dimensional $\Bbbk$-algebra $A$ by $\proj A$ we denote the full subcategory of $\mod A$ consisting of the projective $A$-modules. The following consequence of results of~\cite{ChenYe} will be a crucial tool in proving the main result.

\begin{prop} \label{prop:standard}
Let $A$ be a finite dimensional $\Bbbk$-algebra such that any pseudo-identity $(F, \omega)$ on $\bK^b (\proj A)$ is isomorphic, as a triangle functor, to $(\Id_{\mathbf{K}^b(\proj A)}, \Id_\Sigma)$. Then for any finite dimensional $\Bbbk$-algebra $B$ every derived equivalence between $A$ and $B$ is standard.
\end{prop}

\begin{proof}
According to~\cite{ChenYe}*{Lemma~4.2} our assumption implies that the category $\proj A$ is $\bK$-standard in the sense of~\cite{ChenYe}*{Definition~4.1}. By~\cite{ChenYe}*{Theorem~6.1}, this implies that the category $\mod A$ is $\bD$-standard in the sense of~\cite{ChenYe}*{Definition 5.1}. Now our claim follows from \cite{ChenYe}*{Theorem 5.10}.
\end{proof}

Let $A$ be a finite dimensional $\Bbbk$-algebra. Following~\cite{Vossieck} we say that $A$ is \emph{derived discrete} if for every vector $\bh \in \bbN^\bbZ$ there are only finitely many isomorphism classes of indecomposable objects in $\bD^b (\mod A)$ with cohomology dimension vector $\bh$. Vossieck~\cite{Vossieck}*{Theorem 2.1} has proved that if an algebra $A$ is derived discrete then either $A$ is piecewise hereditary of Dynkin type or $A$ is a one-cycle gentle algebra not satisfying the clock condition. Since the piecewise hereditary algebras have finite global dimension, we may concentrate on one-cycle gentle algebras not satisfying the clock condition. In order to describe representatives of the derived equivalence classes of the derived discrete algebras of infinite global dimension, we need to introduce the following notation: if $n \in \bbN_+$ and $m \in \bbN$, then by $\Lambda (n, m)$ we mean the path algebra of the quiver $Q (n, m)$ of the form
\[
\vcenter{\xymatrix{%
& & & & & \vertexD{1} \ar@/_/[]+L;[ld]+U_{\alpha_0} &
\\
\vertexU{-m} & \vertexU{-m + 1} \ar[l]_{\alpha_{-m}} &
\cdots \ar[l]_{\alpha_{-m + 1}} & \vertexU{-1} \ar[l]_-{\alpha_{-2}} & \vertexL{0} \ar[l]_{\alpha_{-1}}
\ar@/_/[]+D;[rd]+L_{\alpha_{n - 1}} & & \raisebox{-0.5em}[1em][1em]{\vdots} \ar@/_/[]+U;[lu]+R_{\alpha_1}
\\
& & & & & \vertexU{n - 1} \ar@/_/[]+R;[ru]+D_{\alpha_{n - 2}} &}}
\]
bounded by
\[
\alpha_{n - 1} \alpha_0, \, \alpha_{n - 2} \alpha_{n - 1}, \, \ldots, \, \alpha_0 \alpha_1.
\]
Note that $\Lambda (n, 0)$ is the path algebra of the $n$-cycle modulo the ideal generated by the paths of length $2$.

The following proposition is a direct consequence of~\cite{BobinskiGeissSkowronski}*{Theorem~A} (and the remark following it).

\begin{prop} \label{prop:discrete}
Let $A$ be a derived discrete $\Bbbk$-algebra. If the global dimension of $A$ is infinite, then $A$ is derived equivalent to $\Lambda (n, m)$, for some $n \in \bbN_+$ and $m \in \bbN$.
\end{prop}

\section{Objects and morphisms in the homotopy category of $\proj \Lambda (n, m)$} \label{sect:objmorph}

Throughout this section we put $\Lambda := \Lambda (n, m)$ and $Q := Q (n, m)$, for $n \in \bbN_+$ and $m \in \bbN$. The aim of this section is to describe the indecomposable objects of the category $\bK^b (\proj \Lambda)$ and the morphism between them.

\subsection{Objects} \label{subsect:objects}

For $u \in [-m, n - 1]$, denote by $P_u$ the corresponding indecomposable projective $\Lambda$-module, i.e.\ the one whose top is the simple module concentrated at $u$. Similarly, for a path $\sigma$ in $Q$ from $u$ to $v$, we denote by $P_\sigma$ the corresponding map $P_v \to P_u$.

By abuse of language by a path in $\Lambda$ we mean a path in $Q$ which does not contain any of the zero relations $\alpha_{n - 1} \alpha_0$, $\alpha_{n - 2} \alpha_{n - 1}$, \ldots, $\alpha_0 \alpha_1$. For $u \in [-m, n - 1]$ we denote by $\sigma_u$ the maximal path in $\Lambda$ ending at $u$, i.e.\
\[
\sigma_u :=
\begin{cases}
\alpha_u \cdots \alpha_0 & \text{if $u \in [-m, -1]$},
\\
\alpha_u & \text{if $u \in [0, n - 1]$}.
\end{cases}
\]
By $s (u)$ we denote the starting vertex of the path $\sigma_u$. In this way we obtain a map $s \colon [-m, n - 1] \to [-m, n - 1]$ which is given by the formula
\[
s (u) :=
\begin{cases}
1 & \text{if $u \in [-m, -1]$ and $n > 1$},
\\
u + 1 & \text{if $u \in [0, n - 2]$ and $n > 1$},
\\
0 & \text{if $u \in [-m, 0]$ and $n = 1$ or either $u = n - 1$}.
\end{cases}
\]
If $s (u) = s (v)$ and $u \leq v$, then there exists a unique path $\varsigma_{u, v}$ such that $\sigma_u = \varsigma_{u, v} \sigma_v$. Note that $\varsigma_{u, v}$ is the stationary path at $u$, if $v = u$, and $\varsigma_{u, v} = \alpha_u \cdots \alpha_{v - 1}$, if $u < v$. 

Let $\calC$ be the set of quadruples $(k, u, l, v)$ such that $k \in \bbZ$, $l \in \bbN$, $u, v \in [-m, n - 1]$, and either $v = s^l (u)$ or $v < s^l (u) \leq 0$, where by $s^0$ we mean the identity map $[-m, n - 1] \to [-m, n - 1]$. Note that the condition $v < s^l (u) \leq 0$ can be expanded to: either $l = 0$ and $v < u \leq 0$ or $l > 0$ and $v < 0 = s^l (u)$.

For $(k, u, l, v) \in \calC$, let $C_{k, u, l, v}$ be the complex $(C_{k, u, l, v}^i, d_{k, u, l, v}^i)_{i \in \bbZ}$ of projective $\Lambda$-modules defined by:
\begin{gather*}
C_{k, u, l, v}^i :=
\begin{cases}
P_{s^{i - k} (u)} & \text{if either $i \in [k, k + l] \setminus \{ k + l - 1 \}$}
\\
& \quad \text{or $i = k + l - 1$, $v = s^l (u)$, and $l > 0$},
\\
P_{s^{l - 1} (u)} \oplus P_v & \text{if $i = k + l - 1$, $v < s^l (u)$, and $l > 0$},
\\
P_v & \text{if $i = k + l - 1$, $v < s^l (u)$, and $l = 0$},
\\
0 & \text{otherwise},
\end{cases}
\\
\intertext{and}
d_{k, u, l, v}^i :=
\begin{cases}
P_{\sigma_{s^{i - k} (u)}} & \text{if either $i \in [k, k + l - 3]$}
\\
& \quad \text{or $i = k + l - 2$, $v = s^l (u)$, and $l \geq 2$},
\\
& \quad \text{or $i = k + l - 1$, $v = s^l (u)$, and $l \geq 1$},
\\
[P_{\sigma_{s^{l - 2} (u)}}, 0]^{\tr} & \text{$i = k + l - 2$, $v < s^l (u)$, and $l \geq 2$},
\\
[P_{\sigma_{s^{l - 1} (u)}}, P_{\varsigma_{v, s^l (u)}}] & \text{if $i = k + l - 1$, $v < s^l (u)$, and $l \geq 1$},
\\
P_{\varsigma_{v, s^l (u)}} & \text{if $i = k + l - 1$, $v < s^l (u)$, and $l = 0$},
\\
0 & \text{otherwise}.
\end{cases}
\end{gather*}
Following the convention used in~\cite{ArnesenLakingPauksztello} the complex $C_{k, u, l, v}$ can be `unfold' as
\[
u \xrightarrow{\sigma_u^*} s (u) \to \cdots \to s^{l - 1} (u) \xrightarrow{\sigma_{s^{l - 1} (u)}^*} s^l (u),
\]
if $v = s^l (u)$, and
\[
u \xrightarrow{\sigma_u^*} s (u) \to \cdots \to s^{l - 1} (u) \xrightarrow{\sigma_{s^{l - 1} (u)}^*} s^l (u) \xleftarrow{\varsigma_{v, \alpha_{s^l (u)}}^*} v,
\]
if $v < s^l (u)$, where the leftmost module appears in degree $k$ and, if $l = 0$, the part
\[
u \xrightarrow{\sigma_u^*} s (u) \to \cdots \to s^{l - 1} (u) \xrightarrow{\sigma_{s^{l - 1} (u)}^*} s^l (u)
\]
reduces to $u$. In the above we write $w$ instead of $P_w$ and $\sigma^*$ instead of $P_\sigma$. We will sometimes write the above diagrams in a unified way as
\[
u \xrightarrow{\sigma_u^*} s (u) \to \cdots \to s^{l - 1} (u) \xrightarrow{\sigma_{s^{l - 1} (u)}^*} s^l (u) \xymatrix@C=3em@1{\rule{0.5ex}{0pt} & \rule{0.5ex}{0pt} \ar@{-->}[l]_{\varsigma_{v, \alpha_{s^l (u)}}^*}} v,
\]
i.e.\ the dotted arrow denotes an arrow which (together with its staring vertex $v$) may be not present.

The following is the the first step in describing the category $\bK^b (\proj \Lambda)$.

\begin{prop}
The complexes $C_{k, u, l, v}$, $(k, u, l, v) \in \calC$, form a complete set of representatives of the isomorphism classes of the indecomposable objects in $\bK^b (\proj \Lambda)$.
\end{prop}

\begin{proof}
Let $\ol{Q}$ be the \emph{double quiver} of $Q$, i.e.\ the quiver with the same set of vertices as $Q$ and with an additional arrow $\ol{\alpha} \colon v \to u$, for each arrow $\alpha \colon u \to v$ in $Q$. If $\sigma = \beta_1 \cdots \beta_l$ is a path in $\ol{Q}$, then we put $\ol{\sigma} := \ol{\beta}_l \cdots \ol{\beta_1}$, where $\ol{\ol{\alpha}} = \alpha$, for each arrow $\alpha$ of $Q$. Moreover, $\ol{\sigma} := \sigma$, if $\sigma$ is a stationary path.

By a \emph{homotopy string} in $Q$ we mean a path in $\ol{Q}$ which does not contain a path of the form $\alpha \ol{\alpha}$, where $\alpha$ is an arrow in $\ol{Q}$, as a subpath. Homotopy strings $\theta$ and $\theta'$ are called \emph{equivalent} if $\theta' = \theta$ or $\theta' = \ol{\theta}$. Up to this equivalence, the following is a complete list of homotopy strings in $Q$:
\begin{itemize}

\item
stationary paths,

\item
$\alpha_u \cdots \alpha_{u + l}$, $u \in [-m, n - 1]$, $l \in \bbN$,

\item
$\alpha_u \cdots \alpha_{u + l} \ol{\alpha}_{-1} \cdots \ol{\alpha}_v$, $u \in [-m, n - 1]$, $v \in [-m, -1]$, $l \in \bbN$, $u + l \geq 0$, and $u + l \equiv n - 1 \pmod{n}$.

\end{itemize}
In the above formulas, if $w \geq n$, then $\alpha_w$ denotes $\alpha_r$, where $r$ is the reminder of $w$ divided by $n$.

For a pair $(k, \theta)$ consisting of an integer $k$ and a homotopy string $\theta$ one defines (see~\cite{Bobinski2011a}*{Section~3}) a complex $X_{k, \theta}$ in such a way that:
\begin{itemize}

\item
$X_{k, \theta} = C_{k, u, 0, u}$, if $\theta$ is the stationary path at vertex $u$;

\item
$X_{k, \theta} = C_{k + 1, u + l + 1, 0, u}$, if $\theta = \alpha_u \cdots \alpha_{u + l}$, $u \in [-m, -1]$, $l \in \bbN$, and $u + l < 0$;

\item
$X_{k, \theta} = C_{k, u, u + l + 1, v}$, where $v$ is the reminder of $u + l + 1$ divided by $n$, if $\theta = \alpha_u \cdots \alpha_{u + l}$, $u \in [-m, -1]$, $l \in \bbN$, and $u + l \geq 0$,

\item
$X_{k, \theta} = C_{k, u, l + 1, v}$, where $v$ is the reminder of $u + l + 1$ divided by $n$, if $\theta = \alpha_u \cdots \alpha_{u + l}$, $u \in [0, n - 1]$, $l \in \bbN$,

\item
$X_{k, \theta} = C_{k, u, u + l + 1, v}$, if $\theta = \alpha_u \cdots \alpha_{u + l} \ol{\alpha}_{-1} \cdots \ol{\alpha}_v$, $u, v \in [-m, -1]$, $l \in \bbN$, $u + l \geq 0$, and $u + l \equiv n - 1 \pmod{n}$.

\item
$X_{k, \theta} = C_{k, u, l + 1, v}$, if $\theta = \alpha_u \cdots \alpha_{u + l} \ol{\alpha}_{-1} \cdots \ol{\alpha}_v$, $u \in [0, n]$, $v \in [-m, -1]$, $l \in \bbN$, $u + l \geq 0$, and $u + l \equiv n - 1 \pmod{n}$.

\end{itemize}
Observe that all the quadruples $(k, u, l, v) \in \calC$ and (up to the equivalence of the homotopy strings) all the pairs $(k, \theta)$ appear in the above formulas.

Since $\Lambda$ is derived discrete (which corresponds to the fact that, in the terminology of~\cite{Bobinski2011a}, there are no homotopy bands), it follows from~\cite{BekkertMerklen}*{Theorem~3} (see also~\cite{Bobinski2011a}*{Proposition~3.1} for a formulation tailored to our notation) that the complexes $X_{k, \theta}$, where $k \in \bbZ$ and $\theta$ runs through the above representatives of the equivalence classes of the homotopy strings, form a complete set of representatives of the isomorphism classes of the indecomposable objects in $\bK^b (\proj \Lambda)$. Now the above described correspondence between the complexes $X_{k, \theta}$ and the complexes $C_{k, u, l, v}$ implies the claim.
\end{proof}

\subsection{Morphisms}

Our next aim is to describe bases of the morphism spaces between the complexes defined in subsection~\ref{subsect:objects}.


For a quadruple $(k, u, l, v) \in \calC$ let $\Phi_{k, u, l, v}$ be the subset of $\calC$ consisting of the quadruples $(k', u', l', v')$ such that:
\begin{itemize}

\item
$k' \leq k \leq k' + l' \leq k + l$,

\item
$s^{l' + 1} (u') = s^{k' + l' + 1 - k} (u)$ (consequently, $s^i (u') = s^{k' + i - k} (u)$, for each $i \in [k - k' + 1, l']$),

\item
if $k' = k$, then $u \leq u'$, and

\item
if $k + l = k' + l'$ and $v < s^l (u)$, then $v \leq v' < s^l (u)$.

\end{itemize}
In the above situation we denote by $\varphi_{C', C}$, where $C := C_{k, u, l, v}$ and $C' := C_{k', u', l', v'}$, a morphism $\varphi_{C', C} \colon C \to C'$ which we define below via its `unfolded diagram' (in the sense of the convention introduced in~\cite{ArnesenLakingPauksztello}). Note that in all the below diagrams the dashed arrows may not appear.

First, if $k' + l' = k$ and $l > 0$, then $\varphi_{C', C}$ is defined by the diagram
\[
\vcenter{\xymatrix{& & u \ar[r] \ar[d]_{\varsigma_{u, s^{l'} (u')}^*} & \cdots \ar[r] & s^l (u) & & v \ar@{-->}[ll]_-{\varsigma_{v, s^l (u)}^*}
\\
u' \ar[r] & \cdots \ar[r] & s^{l'} (u') & & v' \ar@{-->}[ll]_-{\varsigma_{v', s^{l'} (u')}^*}}}.
\]
Similarly, if $k' + l' = k$, $l = 0$, and $v = u$, then $\varphi_{C', C}$ is given by the diagram
\[
\vcenter{\xymatrix{& & u \ar[d]_{\varsigma_{u, s^{l'} (u')}^*}
\\
u' \ar[r] & \cdots \ar[r] & s^{l'} (u') & & v' \ar@{-->}[ll]_-{\varsigma_{v', s^{l'} (u')}^*}}}.
\]
In the terminology of~\cite{ArnesenLakingPauksztello} these are either singleton single maps (if $u < s^{l'} (u')$) or graph maps (if $u = s^{l'} (u')$).

Next, if $k' + l' = k$, $l = 0$, and $v < u$, then $\varphi_{C', C}$ is represented by the diagram
\[
\vcenter{\xymatrix{& & u \ar[d]_{\varsigma_{u, s^{l'} (u')}^*} & & v \ar[ll]_-{\varsigma_{v, u}^*} \ar[d]^{\varsigma_{v, v'}^*}
\\
u' \ar[r] & \cdots \ar[r] & s^{l'} (u') & & v' \ar[ll]_-{\varsigma_{v', s^{l'} (u')}^*}}}.
\]
Again in the terminology of~\cite{ArnesenLakingPauksztello} it is either a singleton double map (if $u < s^{l'} (u')$ and $v < v'$) or a graph map (otherwise).

Finally assume that if $k' + l' > k$. In this case, if $k + l > k' + l'$, then $\varphi_{C', C}$ is defined by the diagram
\[
\vcenter{\xymatrix{& u \ar[rr]^{\sigma_u^*} \ar[d]_{\varsigma_{u, s^{k - k'} (u')}^*} & & s (u) \ar[r] \ar@{=}[d] & \cdots \ar[r] & s^{k' + l' - k} (u) \ar[r] \ar@{=}[d] & \cdots
\\
\cdots \ar[r] & s^{k - k'} (u') \ar[rr]^{\sigma_{s^{k - k'} (u')}^*} & & s^{k + 1 - k'} (u') \ar[r] & \cdots \ar[r] & s^{l'} (u') & v' \ar@{-->}[l]_-{\varsigma_{v', s^{l'} (u')}^*}}}.
\]
Next, if $k + l = k' + l'$ and $v = s^l (u)$, then $\varphi_{C', C}$ is represented by the diagram
\[
\vcenter{\xymatrix{& u \ar[rr]^{\sigma_u^*} \ar[d]_{\varsigma_{u, s^{k - k'} (u')}^*} & & s (u) \ar[r] \ar@{=}[d] & \cdots \ar[r] & s^{k' + l' - k} (u) \ar@{=}[d]
\\
\cdots \ar[r] & s^{k - k'} (u') \ar[rr]^{\sigma_{s^{k - k'} (u')}^*} & & s^{k + 1 - k'} (u') \ar[r] & \cdots \ar[r] & s^{l'} (u') & v' \ar@{-->}[l]_-{\varsigma_{v', s^{l'} (u')}^*}}}.
\]
Lastly, if $k + l = k' + l'$ and $v < s^l (u)$, then $\varphi_{C', C}$ is given by the diagram
\[
\vcenter{\xymatrix{& u \ar[rr]^{\sigma_u^*} \ar[d]_{\varsigma_{u, s^{k - k'} (u')}^*} & & s (u) \ar[r] \ar@{=}[d] & \cdots \ar[r] & s^l (u) \ar@{=}[d] & v \ar[l]_-{\varsigma_{v, s^l (u)}^*} \ar[d]^{\varsigma_{v, v'}^*}
\\
\cdots \ar[r] & s^{k - k'} (u') \ar[rr]^{\sigma_{s^{k - k'} (u')}^*} & & s^{k + 1 - k'} (u') \ar[r] & \cdots \ar[r] & s^{l'} (u') & v' \ar[l]_-{\varsigma_{v', s^{l'} (u')}^*}}}.
\]
In the three above situations the obtained maps are examples of graph maps.

If $(k, u, l, v), (k', u', l', v') \in \calC$, then we write $(k, u, l, v) \preceq_L (k', u', l', v')$ if one of the following conditions is satisfied:
\begin{enumerate}
\renewcommand{\theenumi}{L\arabic{enumi}}

\item
$k \leq k'$;

\item \label{L2}
$k = k' + 1$ and $u' < u$.

\end{enumerate}
Similarly, we write $(k, u, l, v) \preceq_R (k', u', l', v')$ if one of the following conditions is satisfied:
\begin{enumerate}
\renewcommand{\theenumi}{R\arabic{enumi}}

\item \label{R1}
$k' \leq k + l \leq k' + l'$; moreover, if $v < s^l (u)$, then either $k' < k + l - 1$ or $k' = k + l - 1$ and $v \leq u'$.

\item \label{R2}
$k + l = k' + l' + 1$, $v < s^l (u)$, and either $v' < v$ or $v' = s^{l'} (u')$.

\end{enumerate}
For a quadruple $(k, u, l, v) \in \calC$, we define $\Psi_{k, u, l, v}$ to be the subset of $\calC$ consisting of the quadruples $(k', u', l', v')$ such that
\[
(k, u, l, v) \preceq_L (k', u', l', v'), \quad (k, u, l, v) \preceq_R (k', u', l', v'), \quad \text{and} \quad s^{l + 1} (u) = s^{k + l - k'} (u')
\]
(again the last condition implies that $s^{i + 1} (u) = s^{k + i - k'} (u')$, for all $i \in [k' - k + 1, l - 1]$).

We need to consider some cases in order to define, for $(k, u, l, v) \in \calC$ and $(k', u', l', v') \in \Psi_{k, u, l, v}$, a morphism $\psi_{C', C} \colon C \to C'$, where again $C := C_{k, u, l, v}$ and $C' := C_{k', u', l', v'}$. First, if condition~\eqref{R1} is satisfied and $v = s^l (u)$, then $\psi_{C', C}$ is given by the diagram
\[
\vcenter{\xymatrix{\cdots \ar[r] & s^l (u) \ar[d]^{(-1)^{k + l} \cdot \sigma_{s^l (u)}^*}
\\
\cdots \ar[r] & s^{k + l - k'} (u') \ar[r] & \cdots}}
\]
(all the remaining maps between the components of the complexes $C$ and $C'$ are zero). This map is (up to sign) a singleton single map if either $k + l = k'$ or $k + l = k' + 1$ and $u' < s^l (u)$. In the remaining cases it is (up to sign) a single map representing a quasi-graph map.

Next assume that condition~\eqref{R1} is satisfied and $v < s^l (u)$. Consequently $k' < k + l$ and $s^l (u) \leq 0$. By reversing the diagram representing $C$, we get the following presentation of $\psi_{C', C}$:
\[
\vcenter{\xymatrix{& v \ar[rr]^{\varsigma_{v, s^l (u)}^*} \ar[d]_{(-1)^{k + l} \cdot \varsigma_{v, s^{k + l - k' - 1} (u')}^*} & & s^l (u) \ar[d]^{(-1)^{k + l} \cdot \sigma_{s^l (u)}^*} & \cdots \ar[l]
\\
\cdots \ar[r] & s^{k + l - k' - 1} (u') \ar[rr]^{\sigma_{s^{k + l - k' - 1} (u')}^*} & & s^{k + l - k'} (u') \ar[r] & \cdots}}
\]
(recall that if $k' = k + l - 1$, i.e.\ $s^{k + l - k' - 1} (u') = u'$, then $v \leq u'$). It is either a double map representing a quasi-graph map (if $s^l (u) = s^{k + l - k' - 1} (u')$) or a singleton double map (otherwise) -- again the above statements hold up to sign. Note that the latter situation can only occur when $k' = k + l - 1$, hence $s^{k + l - k' - 1} (u') = u'$, and one deduces that $u' < s^l (u)$ in this case (we use condition~\eqref{L2}, if $l = 0$).

Finally assume that condition~\eqref{R2} is satisfied. In this case the morphism $\psi_{C', C}$ can be represented by the following diagram
\[
\vcenter{\xymatrix{& v \ar[rr]^{\varsigma_{v, s^l (u)}^*} \ar[d]_{(-1)^{k + l} \cdot \varsigma_{v, s^{l'} (u')}^*} & & s^l (u) & \cdots \ar[l]
\\
\cdots \ar[r] & s^{l'} (u') & & v' \ar@{-->}[ll]_-{\varsigma_{v', s^{l'} (u')}^*}}}
\]
It is either a single map representing a quasi-graph map (if $s^l (u) = s^{l'} (u')$) or a singleton single map (otherwise).

The main result of this subsection is the following.

\begin{prop}
The maps $\varphi_{C_{k', u', l', v'}, C_{k, u, l, v}}$, $(k, u, l, v) \in \calC$, $(k', u', l', v') \in \Phi_{k, u, l, v}$, together with the maps $\psi_{C_{k', u', l', v'}, C_{k, u, l, v}}$, $(k, u, l, v) \in \calC$, $(k', u', l', v') \in \Psi_{k, u, l, v}$, form a basis of $\bigoplus_{(k, u, l, v), (k', u', l', v') \in \calC} \Hom (C_{k, u, l, v}, C_{k', u', l', v'})$.
\end{prop}

\begin{proof}
The claim can be verified by lengthy and technical, but relatively straightforward direct calculations, which we leave to the reader. Alternatively, one can use~\cite{ArnesenLakingPauksztello}*{Theorem~3.15} which states that a basis of $\bigoplus_{(k, u, l, v), (k', u', l', v') \in \calC} \Hom (C_{k, u, l, v}, C_{k', u', l', v'})$ is formed by the singleton single maps, the singleton double maps, the graph maps and representatives of the quasi-graph maps. Again we leave a (lengthy and technical) verification that the above list contains all the required maps to the interested reader (one can also use results of~\cite{ArnesenLakingPauksztello}*{Section~7} to facilitate calculations).
\end{proof}

We record some consequences.

\begin{coro} \label{coro:dimension}
Let $(k, u, l, v), (k', u', l', v') \in \calC$, $C := C_{k, u, l, v}$ and $C' := C_{k', u', l', v'}$.
\begin{enumerate}

\item
If $(k', u', l', v') \not \in \Phi_{k, u, l, v} \cup \Psi_{k, u, l, v}$, then $\Hom (C, C') = 0$.

\item
If $(k', u', l', v') \in \Phi_{k, u, l, v} \setminus \Psi_{k, u, l, v}$, then $\dim_\Bbbk \Hom (C, C') = 1$, with basis formed by $\varphi_{C', C}$.

\item
If $(k', u', l', v') \in \Psi_{k, u, l, v} \setminus \Phi_{k, u, l, v}$, then $\dim_\Bbbk \Hom (C, C') = 1$, with basis formed by $\psi_{C', C}$.

\item
If $(k', u', l', v') \in \Phi_{k, u, l, v} \cap \Psi_{k, u, l, v}$, then $\dim_\Bbbk \Hom (C, C') = 2$, with basis formed by $\varphi_{C', C}$ and $\psi_{C', C}$.

\end{enumerate}
\end{coro}

\section{A model of the homotopy category of $\proj \Lambda (n, m)$} \label{sect:category}

Throughout this section we still put $\Lambda := \Lambda (n, m)$, for $n \in \bbN_+$ and $m \in \bbN$. The aim of this section is to present a model of the category $\bK^b (\proj \Lambda)$. We will follow the description presented in~\cite{Bobinski2011b}*{Section~5}, given there without a proof. For convenience of the reader, we include the main steps of the proof below. We also prove necessary facts about irreducible morphisms in $\bK^b (\proj \Lambda)$ here.

\subsection{The category} \label{subsect:category}

We are now ready to describe the full subcategory $\ind \bK^b (\proj \Lambda)$ of $\bK^b (\proj \Lambda)$ formed by the indecomposable objects. We will present it as the quotient $\Bbbk \Gamma / \calI$ of the path category $\Bbbk \Gamma$ of a quiver $\Gamma$ by an ideal $\calI$.

First we put
\[
\Gamma_0 := \{ (i, a, b) : \text{$i \in [0, n - 1]$, $a, b \in \bbZ$, $a \leq b + \delta_{i, 0} \cdot m$} \},
\]
where $\delta_{x, y}$ is the Kronecker delta. For each $V = (i, a, b) \in \Gamma_0$, let
\begin{gather*}
\calF_V := \{ (i, x, y) : \text{$x \in [a, b + \delta_{i, 0} \cdot m]$, $y \in [b, \infty)$} \}
\\
\intertext{and}
\calG_V := \{ (i + 1, x, y) : \text{$x \in (-\infty, a + \delta_{i, n - 1} \cdot m]$, $y \in [a, b + \delta_{i, 0} \cdot m]$} \},
\end{gather*}
where $i + 1$ is calculated modulo $n$. Then
\[
\Gamma_1 := \{ f_{U, V} \colon V \to U : \text{$V \in \Gamma_0$, $U \in \calF_V$, $U \neq V$} \} \cup \{ g_{U, V} \colon V \to U : \text{$V \in \Gamma_0$, $U \in \calG_V$} \}.
\]
Finally, let $\calI$ be the ideal generated by the relations
\begin{align*}
& f_{W, U} \circ f_{U, V} - f_{W, V}, & & V \in \Gamma_0, \, U \in \calF_V, \, U \neq V, \, W \in \calF_U, \, W \neq U,
\\
& g_{W, U} \circ f_{U, V} - g_{W, V}, & & V \in \Gamma_0, \, U \in \calF_V, \, U \neq V, \, W \in \calG_U,
\\
& f_{W, U} \circ g_{U, V} - g_{W, V}, & & V \in \Gamma_0, \, U \in \calG_V, \, W \in \calF_U, \, W \neq U,
\\
& g_{W, U} \circ g_{U, V}, & & V \in \Gamma_0, \, U \in \calG_V, \, W \in \calG_U,
\end{align*}
where $f_{W, V} := 0$ ($g_{W, V} := 0$) if $W \not \in \calF_V$ ($W \not \in \calG_V$ respectively).

We will denote by $f_{V, V}$ the identify morphism $\Id_V$, for $V \in \Gamma_0$. Obviously $f_{U, V} \circ f_{V, V} = f_{U, V}$ = $f_{U, U} \circ f_{U, V}$, $g_{W, V} \circ f_{V, V} = g_{W, V} = f_{W, W} \circ g_{W, V}$, for all $V \in \Gamma_0$, $U \in \calF_U$, $W \in \calG_V$. If $V \in \Gamma_0$ and $U \in \calF_V$, then we put
\[
g_{U, V}' :=
\begin{cases}
g_{U, V} & \text{if $U \in \calG_V$},
\\
0 & \text{otherwise}.
\end{cases}
\]
Observe that $n = 1$ provided $g_{U, V}' \neq 0$.

We will often use the following observation about the morphism spaces in $\Bbbk \Gamma / \calI$ without reference.

\begin{lemm} \label{lemm:dimhom}
Let $V, U \in \Gamma_0$.
\begin{enumerate}
\item
If $U \not \in \calF_V \cup \calG_V$, then $\Hom (V, U) = 0$.
\item
If $U \in \calF_V \setminus \calG_V$, then $\dim_\Bbbk \Hom (V, U) = 1$, with basis formed by $f_{U, V}$.
\item
If $U \in \calG_V \setminus \calF_V$, then $\dim_\Bbbk \Hom (V, U) = 1$, with basis formed by $g_{U, V}$.
\item
If $U \in \calF_V \cap \calG_V$, then $\dim_\Bbbk \Hom (V, U) = 2$, with basis formed by $f_{U, V}$ and $g_{U, V}$.
\end{enumerate}
In particular, if $U \in \calF_V$, then $f_{U, V}$ and $g_{U, V}'$ span $\Hom (V, U)$.
\end{lemm}

\begin{proof}
In the paper, we usually do not distinguish between morphisms in $\Bbbk \Gamma$ (in particular, arrows in $\Gamma$) and the corresponding morphisms in $\Bbbk \Gamma / \calI$, since it should not lead to misunderstanding. However, in this proof, for a morphism $\chi$ in $\Bbbk \Gamma$ we will denote by $[\chi]$ the corresponding morphism in $\Bbbk \Gamma / \calI$.

We introduce a grading in the homomorphism spaces in $\Bbbk \Gamma$ by putting $|f_{U, V}| = 0$ and $|g_{W, V}| = 1$, for $V \in \Gamma_0$, $U \in \calF_V$, and $W \in \calG_U$. Since the relations defining the ideal $\calI$ are homogenous with respect to the above grading, we have the induced grading in the homomorphism spaces in $\Bbbk \Gamma / \calI$. If, for $V, U \in \Gamma_0$ and $d \in \bbN$, $\Hom^{(d)} (V, U)$ denotes the space of morphisms $V \to U$ in $\Bbbk \Gamma / \calI$ of degree $d$ (with respect to the above grading), then we easily see that $\Hom^{(d)} (V, U) = 0$, for $d \geq 2$. Moreover, $\Hom^{(0)} (V, U)$ is either $0$ (if $U \not \in \calF_V$) or is spanned by $[f_{U, V}]$ (if $U \in \calF_V$). Similarly, $\Hom^{(1)} (V, U)$ is either $0$ (if $U \not \in \calG_V$) or is spanned by $[g_{U, V}]$ (if $U \in \calG_V$). Thus in order to finish the proof, we need to show that $[f_{U, V}] \neq 0$ and $[g_{W, V}] \neq 0$, for all $V \in \Gamma_0$, $U \in \calF_V$, and $W \in \calG_U$.

Fix $V \in \Gamma_0$ and $U \in \calF_V$. Let $\calP := \calP_{U, V}$ be the set of the sequences $(W_0, \ldots, W_l)$ such that $W_0 = V$, $W_l = U$, and $W_p \in \calF_{W_{p - 1}}$ and $W_p \neq W_{p - 1}$, for all $p \in [1, l]$. Note that the assumption $U \in \calF_V$ implies that $W_p \in \calF_{W_q}$, for all $q \leq p$. For $(W_0, \ldots, W_l) \in \calP$, we put
\[
\omega_{W_0, \ldots, W_l} := f_{W_l, W_{l - 1}} \circ \cdots \circ f_{W_1, W_0}
\]
(in particular, if $V = U$, then $\omega_{V} := \Id_V$). If, moreover, $l \geq 2$ and $p \in [0, l - 2]$, then we put
\[
\rho_{W_0, \ldots, W_l, p} := \omega_{W_0, \ldots, W_l} - \omega_{W_0, \ldots, W_p, W_{p + 2}, \ldots, W_l}.
\]

Let $(\Bbbk \Gamma)^{(0)} (V, U)$ the space of morphisms $V \to U$ in $\Bbbk \Gamma$ of degree $0$ and $\calI^{(0)} (V, U) := (\Bbbk \Gamma)^{(0)} (V, U) \cap \calI$. One easily observes that $(\Bbbk \Gamma)^{(0)} (V, U)$ is spanned by the morphisms $\omega_{W_0, \ldots, W_l}$, for $(W_0, \ldots, W_l) \in \calP$, while $\calI^{(0)} (V, U)$ is spanned by the morphisms $\rho_{W_0, \ldots, W_l, p}$, for $(W_0, \ldots, W_l) \in \calP$ with $l \geq 2$ and $p \in [0, l - 2]$. In fact, we can find a smaller spanning set of $\calI^{(0)} (V, U)$. Namely, $\calI^{(0)} (V, U)$ is spanned by the morphisms $\rho_{W_0, \ldots, W_l, 0}$, for $(W_0, \ldots, W_l) \in \calP$ with $l \geq 2$, which follows from the equality
\[
\rho_{W_0, \ldots, W_l, p} = \sum_{q \in [1, p + 1]} \rho_{W_0, W_q, \ldots W_l, 0} - \sum_{q \in [1, p]} \rho_{W_0, W_q, \ldots, W_p, W_{p + 2}, \ldots, W_l, 0}.
\]
It is an easy exercise in linear algebra to show that
\[
\calI^{(0)} (V, U) + \Bbbk \cdot f_{U, V} = (\Bbbk \Gamma)^{(0)} (V, U) \qquad \text{and} \qquad \calI^{(0)} (V, U) \cap \Bbbk \cdot f_{U, V} = 0
\]
(we use here, that if we order the elements of $\calP$ in such a way that longer sequences precede shorter ones, then $\omega_{W_0, \ldots, W_l}$ is the leading term of $\rho_{W_0, \ldots, W_l, 0}$). Now the Second Isomorphism Theorem implies that the canonical map $\Bbbk \cdot f_{U, V} \to (\Bbbk \Gamma)^{(0)} (V, U) / \calI^{(0)} (V, U)$ is an isomorphism, hence in particular $[f_{U, V}] \neq 0$.

The proof that $[g_{W, V}] \neq 0$, for all $V \in \Gamma_0$ and $W \in \calG_V$, is analogous.
\end{proof}

The following proposition is the main result of this subsection and describes the aforementioned equivalence between $\Bbbk \Gamma / \calI$ and $\ind \bK^b (\proj \Lambda)$.

\begin{prop} \label{prop:category}
Let
\begin{multline*}
\Theta (i, p \cdot (m + n) + r + i, q \cdot (m + n) + t + i - \delta_{i, 0} \cdot m)
\\
:= \begin{cases}
C_{- q \cdot n - t - i, -t, (q - p) \cdot n + (t - r), -r} & \text{if $r \in [-n + 1, 0]$ and $t \in [- n + 1, -1]$},
\\
C_{- q \cdot n - i, - m + t, (q - p) \cdot n - r, -r} & \text{if $r \in [- n + 1, 0]$, $t \in [0, m]$},
\\
& \qquad \text{and $(q - p) \cdot n - r > 0$},
\\
C_{- q \cdot n - i, - m + t, 0, -m + t} & \text{if $r \in [- n + 1, 0]$, $t \in [0, m]$},
\\
& \qquad \text{and $(q - p) \cdot n - r = 0$},
\\
C_{- q \cdot n - t - i, - t, (q - p) \cdot n + t, - m - 1 + r} & \text{if $r \in [1, m]$ and $t \in [-n + 1, -1]$},
\\
C_{- q \cdot n - i, - m + t, (q - p) \cdot n, - m - 1 + r} & \text{if $r \in [1, m]$ and $t \in [0, m]$},
\end{cases}
\end{multline*}
for $i \in [0, n  - 1]$, $p, q \in \bbZ$, $r, t \in [- n + 1, m]$, such that $p \cdot (m + n) + r \leq q \cdot (m + n) + t$. If
\[
\Theta (f_{U, V}) := \varphi_{\Theta (U), \Theta (V)} \qquad \text{and} \qquad \Theta (g_{W, V}) := \psi_{\Theta (W), \Theta (V)},
\]
for $V \in \Gamma_0$, $U \in \calF_V$, $W \in \calG_V$, then $\Theta$ induces an equivalence $\Bbbk \Gamma / \calI \xrightarrow{\sim} \ind \bK^b (\proj \Lambda)$.
\end{prop}

Before giving the proof we explain how $\Theta$ on vertices of $\Gamma$ is constructed. One gets from \cite{Bobinski2011a}*{Corollary~6.3} (see also~\cite{BobinskiGeissSkowronski}*{Lemma~3.1}) that for each $i \in \{ -m \} \cup [1, n - 1]$ there exists a unique (up to scalar and radical square) irreducible map in $\bK^b (\proj \Lambda)$ starting at $P_i$ (which we identify with $C_{0, i, 0, i}$). On the other hand, we will prove (see~Corollary~\ref{coro:irr}) that there is a unique (up to scalar and radical square) irreducible map in $\Bbbk \Gamma / \calI$ starting at $V$ if and only if $V = (i, a, a - \delta_{i, 0} \cdot m)$, for some $i \in [0, n - 1]$ and $a \in \bbZ$. Thus we can start defining  $\Theta$ by putting $\Theta (0, 0, -m) := C_{0, -m, 0, -m}$. Observe that there are nonzero maps $P_{-m} \to P_1$ and $P_{i - 1} \to P_i$, $i \in [2, n- 1]$, provided $n > 1$. Since $(i, 0, 0)$ and $(i - 1, 0, - \delta_{i - 1, 0} \cdot m)$ are the only objects of the form $(j, a, a - \delta_{j, 0} \cdot m)$ with nonzero maps from $(i - 1, 0, - \delta_{i - 1, 0} \cdot m)$, by easy induction it follows that we have to put $\Theta (i, 0, 0) := C_{0, i, 0, i}$, $i \in [1, n - 1]$.

We will show in Corollary~\ref{coro:irreducible} that an irreducible map starting at $(i, a, b)$ terminates at either $(i, a, b + 1)$ or $(i, a + 1, b)$ (note that the latter object may not exist). On the other hand, using~\cite{Bobinski2011a}*{Main Theorem, Part II} we can obtain a similar result about irreducible maps in $\bK^b (\proj \Lambda)$. More precisely, if $C := C_{k, u, l, v}$ then there is an irreducible map $C \to C'$, where
\[
C' :=
\begin{cases}
C_{k - 1, u - 1, l + 1, v} & \text{if $u > 1$},
\\
C_{k - 1, -m, l + 1, v} & \text{if either $n > 1$ and $u = 1$ or $n = 1$ and $u = 0$},
\\
C_{k - 1, n - 1, l + 1, v} & \text{if $n > 1$ and $u = 0$},
\\
C_{k, u + 1, l, v} & \text{if $u < 0$}.
\end{cases}
\]
Moreover,
unless $l = 0$ and either $u = v \in \{ -m \} \cup [1, n - 1]$ or $v + 1 = u$, we also have an irreducible map $C \to C''$, where
\[
C'' :=
\begin{cases}
C_{k, u, l - 1, v - 1} & \text{if $l > 0$ and $v > 0$},
\\
C_{k, u, l, -m} & \text{if either $l > 0$, $v = 0$, and $m > 0$},
\\
& \qquad \text{or $l = 0$ and $-m < u = v \leq 0$},
\\
C_{k, u, l - 1, n - 1} & \text{if $l > 0$ and either $v = 0$ and $m = 0$ or $v = -1$},
\\
C_{k, u, l, v + 1} & \text{if either $l > 0$ and $v < -1$ or $l = 0$ and $v + 1 < u$}.
\end{cases}
\]
Using the above observations, one can deduce the required formula for $\Theta$.

\begin{proof}[Proof of Proposition~\ref{prop:category}]
Observe that if $p, q \in \bbZ$, $r, t \in [-n + 1, 0]$, then $p \cdot (m + n) + r \leq q \cdot (m + n) + t$ if and only if $p \cdot n + r \leq q \cdot n + t$. As a first consequence we obtain that $\Theta$ is well-defined on the vertices of $\Gamma$. Secondly, we use the above observation to show that, for each $V \in \Gamma_0$,
\begin{equation} \label{eq:category}
\Theta (\calF_V) = \Phi_{\Theta (V)} \qquad \text{and} \qquad \Theta (\calF_V) = \Psi_{\Theta (V)},
\end{equation}
where $\Phi_{C_{k, u, l, v}} := \Phi_{k, u, l, v}$ and $\Psi_{C_{k, u, l, v}} := \Psi_{k, u, l, v}$. This implies in particular that $\Theta$ is well-defined on the arrows of $\Gamma$.  Since obviously $\Theta (\Id_V) = \Theta (f_{V, V}) = \varphi_{\Theta (V), \Theta (V)} = \Id_{\Theta (V)}$, we get a functor $\Bbbk \Gamma \to \ind \bK^b (\proj \Lambda)$, which we also denote by $\Theta$. Again by checking the indices we see that $\Theta$ is dense. Moreover, formulas~\eqref{eq:category} imply that $\Theta$ is full.


In order to show that $\Theta$ induces a functor $\Bbbk \Gamma / \calI \to \ind \bK^b (\proj \Lambda)$, we have to verify that the relations defining the ideal $\calI$ are satisfied, i.e.\
\begin{align*}
\varphi_{\Theta (W), \Theta (U)} \circ \varphi_{\Theta (U), \Theta (V)} & = \varphi_{\Theta (W), \Theta (V)}, & & V \in \Gamma_0, \, U \in \calF_V, \, W \in \calF_U, \,
\\
\psi_{\Theta (W), \Theta (U)} \circ \varphi_{\Theta (U), \Theta (V)} & = \psi_{\Theta (W), \Theta (V)}, & & V \in \Gamma_0, \, U \in \calF_V, \, W \in \calG_U,
\\
\varphi_{\Theta (W), \Theta (U)} \circ \psi_{\Theta (U), \Theta (V)} & = \psi_{\Theta (W), \Theta (V)}, & & V \in \Gamma_0, \, U \in \calG_V, \, W \in \calF_U, \,
\\
\psi_{\Theta (W), \Theta (U)} \circ \psi_{\Theta (U), \Theta (V)} & = 0, & & V \in \Gamma_0, \, U \in \calG_V, \, W \in \calG_U,
\end{align*}
where $\varphi_{C', C} := 0$ ($\psi_{C', C} := 0$) if $C' \not \in \Phi_C$ ($C' \not \in \Psi_C$, respectively). This is done by simple, but lengthy case by case analysis, so we only give an example here, which will also illustrate that in general these relations hold only up to homotopy.

Take $V \in \Gamma_0$, $U \in \calF_V$ and $W \in \calG_U$, and write
\begin{gather*}
V = (i, p \cdot (m + n) + r + i, q \cdot (m + n) + t + i - \delta_{i, 0} \cdot m),
\\
U = (i, p' \cdot (m + n) + r' + i, q' \cdot (m + n) + t' + i - \delta_{i, 0} \cdot m),
\intertext{and}
W = (j, p'' \cdot (m + n) + r'' + j, q'' \cdot (m + n) + t'' + j - \delta_{j, 0} \cdot m),
\end{gather*}
for $i \in [0, n - 1]$, $p, q, p', q', p'', q'' \in \bbZ$, $r, t, r', t', r'', t'' \in [-n + 1, m]$, where $j := i + 1$, if $i < n - 1$, and $j := 0$, if $i = n - 1$. Assume also that $r, r', r'' \in [-n + 1, 0]$ and $t, t', t'' \in [-n + 1, -1] \cup \{ m \}$. Put
\begin{gather*}
u :=
\begin{cases}
-t & \text{if $t < 0$},
\\
0 & \text{if $t = m$},
\end{cases}
\quad
u' :=
\begin{cases}
-t' & \text{if $t' < 0$},
\\
0 & \text{if $t' = m$},
\end{cases}
\quad
u'' :=
\begin{cases}
-t'' & \text{if $t'' < 0$},
\\
0 & \text{if $t'' = m$},
\end{cases}
\\
k = - q \cdot n + u - i, \quad l := (q - p) \cdot n - u - r,
\\
k' = - q' \cdot n + u' - i, \quad l' := (q' - p') \cdot n - u' - r',
\\
k'' =
\begin{cases}
- q'' \cdot n + u'' - i - 1 & \text{if $i < n - 1$},
\\
- q'' \cdot n + u'' & \text{if $i = n - 1$},
\end{cases} \quad l'' := (q'' - p'') \cdot n - u'' - r''.
\end{gather*}
In this case the composition $\psi_{\Theta (W), \Theta (U)} \circ \varphi_{\Theta (U), \Theta (V)}$ is given by the diagram
\[
\vcenter{\xymatrix{u \ar[r] & \cdots \ar[rr] & & s^{k' + l' - k} (u) \ar[rr]^-{\sigma_{s^{k' + l' - k} (u)}^*} \ar[d]^{(-1)^{k' + l'} \cdot \sigma_{s^{k' + l' - k} (u)}^*} & & \cdots \ar[r] & s^l (u)
\\
u'' \ar[r] & \cdots \ar[rr]^-{\sigma_{s^{k' + l' - k} (u)}^*} & & s^{k' + l' - k''} (u'') \ar[rr] & & \cdots \ar[r] & s^{l''} (u'')}}.
\]
If $W \not \in \calG_U$, then $k'' < k$ or $k'' + l'' < k + l$, and the above map is homotopic to $0$. Otherwise, it is homotopic to
\[
\vcenter{\xymatrix{u \ar[r] & \cdots \ar[r] & s^l (u) \ar[d]^{(-1)^{k + l} \cdot \sigma_{s^l (u)}^*}
\\
u'' \ar[r] & \cdots \ar[r] & s^{k + l - k''} (u'') \ar[r] & \cdots \ar[r] & s^{l''} (u'')}},
\]
which is $\psi_{\Theta (W), \Theta (V)}$.

By abuse of notation denote by $\Theta$ the induced functor $\Bbbk \Gamma / \calI \to \ind \bK^b (\proj \Lambda)$. We already know that $\Theta$ is dense and full. Moreover, formula~\eqref{eq:category}, Lemma~\ref{lemm:dimhom} and Corollary~\ref{coro:dimension} imply that $\Theta$ is also faithful. Indeed, the above-mentioned facts imply that $\dim_\Bbbk \Hom (V, U) = \dim_\Bbbk \Hom (\Phi (V), \Phi (U)) < \infty$ for all objects $V$ and $U$ of $\Bbbk \Gamma / \calI$. Since we already know that, for each $V$ and $U$, the map $\Hom (V, U) \to \Hom (\Phi (V), \Phi (U))$ induced by the functor $\Phi$ is surjective, it also has to be injective as well. Consequently, $\Theta$ is an equivalence.
\end{proof}

From now on we will identify $\Bbbk \Gamma / \calI$ with its image under this equivalence, treat the vertices of $\Gamma$ as complexes of projective $\Lambda$-modules. Since $\Sigma C_{k, u, l, v} = C_{k - 1, u, l, v}$, one easily checks (using the formula for $\Theta$ from Proposition~\ref{prop:category}) that $\Sigma V = (i + 1, a + 1 + \delta_{i, n - 1} \cdot m, b + 1 + \delta_{i, 0} \cdot m)$ provided $V = (i, a, b)$.

As pseudo-identities act trivially on morphisms between projective modules, it is important to identify the vertices of $\Gamma$ corresponding to the modules $P_i$, $i \in [-m, n + 1]$. The following is an easy consequence of Proposition~\ref{prop:category}.

\begin{coro} \label{coro:projective}
Up to the equivalence $\Theta$ described in Proposition~\ref{prop:category} we have: $P_i = (0, 0, i)$, $i \in [-m, 0]$, and $P_i = (i, 0, 0)$, $i \in [1, n - 1]$.
\end{coro}

\begin{proof}
This follows immediately from the formula for $\Theta$.
\end{proof}

\subsection{Irreducible morphisms}

As a final step in this section we describe the irreducible morphisms between the objects of $\Bbbk \Gamma / \calI$. We start with the following easy observation, whose proof is left to the reader.


\begin{lemm} \label{lemm:auto}
Let $V, U \in \Gamma_0$ and $f \in \Hom (V, U)$. Then $f$ is an isomorphism if and only if $U = V$ and $f = \lambda \cdot f_{V, V} + \mu \cdot g_{V, V}'$, for some $\lambda, \mu \in \Bbbk$, $\lambda \neq 0$. \qed
\end{lemm}

Recall that by the \emph{radical} of a Krull--Schmidt category we mean the ideal consisting of the maps $f \colon X \to Y$ such that, for each split monomorphism $\iota \colon X' \to X$ and each split epimorphism with $\pi \colon Y \to Y'$, with $X'$ and $Y'$ indecomposable, $\pi \circ f \circ \iota$ is not an isomorphism. As a first immediate consequence of Lemma~\ref{lemm:auto} we get the following (recall the we treat the category $\Bbbk \Gamma / \calI$ as a subcategory $\bK^b (\proj \Lambda)$).

\begin{coro} \label{coro:radical}
The radical of the category $\bK^b (\proj \Lambda)$ coincides with the ideal $\langle \Gamma_1 \rangle$ generated by the arrows in $\Gamma$. \qed
\end{coro}

For each arrow $\gamma$ in $\Gamma_1$ we define its \emph{degree} $\deg \gamma$ in the following way. If $V = (i, a, b) \in \Gamma_0$ and $U = (j, x, y) \in \calF_V$ (in particular, $j = i$), $U \neq V$, then we put $\deg f_{U, V} := (x - a) + (y - b)$. Moreover, if $V \in \Gamma_0$ and $U \in \calG_V$, then $\deg g_{U, V} := \infty$. These definitions extend naturally to the paths in $\Gamma$ (trivial paths having degree $0$). We have the following.

\begin{lemm} \label{lemm:rad}
Let $V \in \Gamma_0$. If $U \in \calF_V$, then $f_{U, V} \in \rad^{\deg f_{U, V}} (V, U) \setminus \rad^{\deg f_{U, V} + 1} (V, U)$. Similarly, if $U \in \calG_V$, then $g_{U, V} \in \rad^\infty (V, U)$.
\end{lemm}

\begin{proof}
In this proof, similarly as in the proof of Lemma~\ref{lemm:dimhom}, we will distinguish between morphisms in $\Bbbk \Gamma$ and $\Bbbk \Gamma / \calI$. In particular, for a morphism $\chi$ in $\Bbbk \Gamma$ we will denote by $[\chi]$ the corresponding morphism in $\Bbbk \Gamma / \calI$.

For $V, U \in \Gamma_0$ and $k \in \bbN \cup \{ \infty \}$, let $(\Bbbk \Gamma)_k (V, U)$ be the subspace of $\Hom_{\Bbbk \Gamma} (V, U)$ spanned by the paths of degree $k$. Obviously, $\Hom_{\Bbbk \Gamma} (V, U) = \bigoplus_{k \in \bbN \cup \{ \infty \}} (\Bbbk \Gamma)_k (V, U)$. Observe that the ideal $\calI$ is homogenous with respect to the above grading, i.e.\ if we put $\calI_k (V, U) := (\Bbbk \Gamma)_k (V, U) \cap \calI (V, U)$, then $\calI (V, U) = \bigoplus_{k \in \bbN \cup \{ \infty \}} \calI_k (V, U)$. Consequently, $\Hom_{\Bbbk \Gamma / \calI} (V, U) = \bigoplus_{k \in \bbN \cup \{ \infty \}} \calR_k (V, U)$, where $\calR_k (V, U) := (\Bbbk \Gamma)_k (V, U) / \calI_k (V, U)$, for $k \in \bbN \cup \{ \infty \}$.

We know from Corollary~\ref{coro:radical} that $\rad_{\Bbbk \Gamma / \calI} (V, U) = \bigoplus_{k \in \bbN_+ \cup \{ \infty \}} \calR_k (V, U)$. Now it is fairly easy to show that $\rad_{\Bbbk \Gamma / \calI}^d (V, U) = \bigoplus_{k \geq d} \calR_k (V, U)$, for each $d \in \bbN_+ \cup \{ \infty \}$ -- in order to prove this, we use two easy observations: first, if $\gamma \in \Gamma_1$ and $\deg \gamma < \infty$, then the map $[\gamma]$ is a composition in $\Bbbk \Gamma / \calI$ of $\deg \gamma$ morphisms (arrows) of degree $1$; secondly, if $\gamma \in \Gamma_1$ and $\deg \gamma = \infty$, then there exist arrows $\gamma', \gamma'' \in \Gamma_1$ such that $\deg \gamma' = \infty$, $\deg \gamma'' = 1$, and $[\gamma] = [\gamma'] \circ [\gamma'']$.

If $U \in \calF_V$, then by definition $[f_{U, V}] \in \calR_{\deg f_{U, V}} (V, U)$ and $[f_{U, V}] \not \in \bigoplus_{k > \deg_{f_{U, V}}} \calR_k (V, U)$ (here we use that $[f_{U, V}] \neq 0$ by Lemma~\ref{lemm:dimhom}). Similarly, if $U \in \calG_V$, then $[g_{U, V}] \in \calR_\infty (U, V)$. This finishes the proof.
\end{proof}


Since a morphism $f$ between indecomposable objects $X$ and $Y$ is irreducible if and only if $f \in \rad (X, Y) \setminus \rad^2 (X, Y)$ (see~\cite{Ringel1984}*{subsection~2.2}), Lemma~\ref{lemm:rad} implies the following.

\begin{coro} \label{coro:irreducible}
Let $V = (i, a, b), U \in \Gamma_0$, and $f \in \Hom (V, U)$. Then $f$ is irreducible if and only if either $U = (i, a, b + 1)$ or $U = (i, a + 1, b)$, and $f = \lambda \cdot f_{U, V} + \mu \cdot g_{U, V}'$, for some $\lambda, \mu \in \Bbbk$, $\lambda \neq 0$. \qed
\end{coro}

If $X$ and $Y$ are indecomposable objects, then we put $\Irr (X, Y) := \rad (X, Y) / \rad^2 (X, Y)$. We obtain the following from the above considerations.

\begin{coro} \label{coro:irr}
Let $V \in \Gamma_0$.
\begin{enumerate}

\item
If $U \in \Gamma_0$, then $\dim_\Bbbk \Irr (V, U) \leq 1$.

\item
There are at most two $U \in \Gamma_0$ such that $\Irr (V, U) \neq 0$.

\item \label{point:irr3}
There is a unique object $U \in \Gamma_0$ such that $\Irr (V, U) \neq 0$ if and only if $V = (i, a, a - \delta_{i, 0} \cdot m)$, for some $i \in [0, n - 1]$, $a \in \bbZ$. \qed

\end{enumerate}
\end{coro}

We conclude this subsection with a proposition which will play a crucial role in constructing an isomorphism between $F$ and $\Id_{\bK^b (\proj \Lambda)}$.

\begin{prop} \label{prop:auto}
Let $V = (i, a, b) \in \Gamma_0$.
\begin{enumerate}

\item \label{point:auto1}
If $U = (i, a, b + 1)$ and $f \colon V \to U$ is irreducible, then there exists an automorphism $\phi \in \Aut (U)$ such that $\phi \circ f = f_{U, V}$.

\item \label{point:auto2}
If $b = a + 1 - \delta_{i, 0} \cdot m$, $U = (i, a + 1, a + 1 - \delta_{i, 0} \cdot m)$, and $f \colon V \to U$ is irreducible with $f \circ f_{V, W} = 0$, where $W := (i, a, a - \delta_{i, 0})$, then there exists an automorphism $\phi \in \Aut (U)$ such that $\phi \circ f = f_{U, V}$.

\end{enumerate}
Dually:
\begin{enumerate}
\renewcommand{\theenumi}{\arabic{enumi}'}

\item \label{point:auto1prim}
If $U = (i, a - 1, b)$ and $f \colon U \to V$ is irreducible, then there exists an automorphism $\phi \in \Aut (U)$ such that $f \circ \phi = f_{U, V}$.

\item \label{point:auto2prim}
If $b = a + 1 - \delta_{i, 0} \cdot m$, $U = (i, a, a - \delta_{i, 0} \cdot m)$, and $f \colon U \to V$ is irreducible with $f_{W, V} \circ f = 0$, where $W := (i, a + 1, a + 1 - \delta_{i, 0})$, then there exists an automorphism $\phi \in \Aut (U)$ such that $f \circ \phi = f_{U, V}$.

\end{enumerate}
\end{prop}

\begin{proof}
\eqref{point:auto2}~We know from Corollary~\ref{coro:irreducible} that $f = \lambda \cdot f_{U, V} + \mu \cdot g_{U, V}'$, for some $\lambda, \mu \in \Bbbk$, $\lambda \neq 0$. We show that the condition $f \circ f_{V, W} = 0$ implies $\mu \cdot g_{U, V}' = 0$. Indeed, assume that $g_{U, V}' \neq 0$, i.e.\ $U \in \calG_V$. This means $n = 1$, $i = 0$, and
\[
a + 1 \leq a + m \qquad \text{and} \qquad a \leq a + 1 - m \leq a + 1,
\]
which implies $m = 1$. Consequently, $U \in \calG_W$, thus $g_{U, V}' \circ f_{V, W} = g_{U, V} \circ f_{V, W} = g_{U, W} \neq 0$. Since $U \not \in \calF_W$, we get $0 = f \circ f_{V, W} = \mu \cdot g_{U, W}$, hence $\mu = 0$, thus $\mu \cdot g_{U, V}' = 0$.

As $\mu \cdot g_{U, V}' = 0$, we can put $\phi := \lambda^{-1} \cdot f_{U, U}$, and get $\phi \in \Aut (U)$ such that $\phi \circ f = f_{U, V}$.

\eqref{point:auto1}~Again $f = \lambda \cdot f_{U, V} + \mu \cdot g_{U, V}'$, for some $\lambda, \mu \in \Bbbk$, $\lambda \neq 0$, by Corollary~\ref{coro:irreducible}. We put
$\phi := \lambda^{-1} \cdot f_{U, U}  - \lambda^{-2} \cdot \mu \cdot g_{U, U}'$. Then $\phi \in \Aut (U)$ (by Lemma~\ref{lemm:auto}) and by direct calculations (similarly as above we show that $g_{U, U}' \neq 0$ provided $g_{U, V}' \neq 0$) we get $\phi \circ f = f_{U, V}$.

Proofs of statements~\eqref{point:auto1prim} and~\eqref{point:auto2prim} are dual.
\end{proof}

\section{Proof of the main result} \label{sect:proof}

Throughout this section again $\Lambda := \Lambda (n, m)$, for $n \in \bbN_+$ and $m \in \bbN$. We also fix a pseudo-identity $(F, \omega)$ on $\bK^b (\proj \Lambda)$. Our aim is to show that $(F, \omega)$ is isomorphic, as a triangle functor, to the identity functor $(\Id_{\mathbf{K}^b(\proj \Lambda)}, \Id_\Sigma)$ on $\bK^b (\proj \Lambda)$, where $\Sigma$ is (as usual) the suspension functor.

We will freely use notation introduced in Section~\ref{sect:category}. In particular, by $\Gamma$ we mean the quiver introduced in subsection~\ref{subsect:category}.


\subsection{Construction of the isomorphism}

In this subsection we construct automorphisms $\phi_U \in \Aut (U)$, $U \in \Gamma_0$, which satisfy (as we will show in the next subsection) the condition $\phi_U \circ F (f) = f \circ \phi_V$, for all $V, U \in \Gamma_0$ and morphisms $f \colon V \to U$. We will then use these automorphisms to construct a natural isomorphism between $(F, \omega)$ and $(\Id_{\bK^b (\proj \Lambda)}, \omega')$, for some connecting isomorphism $\omega' \colon \Sigma \to \Sigma$.

Before we proceed with the construction, we explain how our method differs from the one used by Chen and Zhang in~\cite{ChenZhang}. In the cases considered in~\cite{ChenZhang}, for each indecomposable object $U$ of the homotopy category which is not (homotopy equivalent to) a stalk complex, there exists a stalk complex $V$ such that $\dim_{\Bbbk} \Hom (V, U) = 1$. Consequently, if $f \colon V \to U$ is a nonzero map, then $F (f) = \lambda \cdot f$, for some scalar $\lambda$, and they use $\lambda$ to define the required automorphism $\phi_U$.

In our case, the above property may not hold. Thus in order to construct $\phi_U$, we use an inductive procedure based on Proposition~\ref{prop:auto}, which uses a properly chosen irreducible morphism starting or terminating at $U$.

Fix $i \in [0, n - 1]$. As mentioned above we construct automorphisms $\phi_U$ for $U = (i, a, b)$, $a, b \in \bbZ$, $a \leq b + \delta_{i, 0} \cdot m$, inductively.

First we put $\phi_U := f_{U, U} (= \Id_U)$, for $U = (i, 0, b)$, $b \in [-\delta_{i, 0} \cdot m, 0]$. Note that if $i = 0$ and $j \in [-m, -1]$, then $\phi_U \circ F (f_{U, V}) \circ \phi_V^{-1} = f_{U, V}$, where $V := (0, 0, j)$ and $U := (0, 0, j + 1)$, since $F (f_{U, V}) = f_{U, V}$, as $F$ is a pseudo-identity and $V$ and $U$ are projective modules (by Corollary~\ref{coro:projective}).

Next assume that $b \geq 0$, $U = (i, 0, b + 1)$, and $\phi_V$, where $V := (i, 0, b)$, is already constructed. Since $f_{U, V}$ is irreducible by Corollary~\ref{coro:irreducible}, $F$ is an equivalence, and $\phi_V$ is an automorphism, $F (f_{U, V}) \circ \phi_V^{-1}$ is irreducible. Using Proposition~\ref{prop:auto}\eqref{point:auto1} we find an automorphism $\phi_U \in \Aut (U)$ such that $\phi_U \circ F (f_{U, V}) \circ \phi_V^{-1} = f_{U, V}$.

Now assume that $a \geq 0$, $U = (i, a + 1, a + 1 - \delta_{i, 0} \cdot m)$, and $\phi_V$ and $\phi_W$ such that $\phi_V \circ F (f_{V, W}) \circ \phi_W^{-1} = f_{V, W}$, where $V := (i, a, a + 1 - \delta_{i, 0} \cdot m)$ and $W := (i, a, a - \delta_{i, 0} \cdot m)$, are already known. Similarly as above $F (f_{U, V}) \circ \phi_V^{-1}$ is irreducible. Moreover,
\begin{multline*}
F (f_{U, V}) \circ \phi_V^{-1} \circ f_{V, W} = F (f_{U, V}) \circ \phi_V^{-1} \circ \phi_V \circ F (f_{V, W}) \circ \phi_W^{-1}
\\
= F (f_{U, V} \circ f_{V, W}) \circ \phi_W^{-1} = F (0) \circ \phi_W^{-1} = 0,
\end{multline*}
as $U \not \in \calF_W$. Consequently there exists an automorphism $\phi_U \in \Aut (U)$ such that $\phi_U \circ F (f_{U, V}) \circ \phi_V^{-1} = f_{U, V}$ by Proposition~\ref{prop:auto}\eqref{point:auto2}.

Finally assume that $a > 0$, $b \geq a - \delta_{i, 0} \cdot m$, $U = (i, a, b + 1)$, and $\phi_V$, where $V := (i, a, b)$, is already constructed. In this case we use Proposition~\ref{prop:auto}\eqref{point:auto1} again in order to find an automorphism $\phi_U \in \Aut (U)$ such that $\phi_U \circ F (f_{U, V}) \circ \phi_V^{-1} = f_{U, V}$.

The automorphisms $\phi_U$, for $U = (i, a, b)$ with $a < 0$, are defined similarly. Namely, for a given $a < 0$ we first define $\phi_U$, for $U = (i, a, a + 1 - \delta_{i, 0} \cdot m)$, using Proposition~\ref{prop:auto}\eqref{point:auto1prim} and that $\phi_V$, where $V := (i, a + 1, a + 1 - \delta_{i, 0} \cdot m)$, is already defined. Next we construct $\phi_U$, for $U = (i, a, a - \delta_{i, 0} \cdot m)$ using Proposition~\ref{prop:auto}\eqref{point:auto2prim}, and finally $\phi_U$, for $U = (i, a, b)$ with $b > a + 1 - \delta_{i, 0} \cdot m$, using Proposition~\ref{prop:auto}\eqref{point:auto1} and induction on $b$.

\subsection{Verification}

We verify now that $\phi_U \circ F (f) = f \circ \phi_V$ (equivalently, $F' (f) = f$, where $F' (f) := \phi_U \circ F (f) \circ \phi_V^{-1}$), for all $V, U \in \Gamma_0$ and morphisms $f \colon V \to U$, where $\phi_U$, $U \in \Gamma_0$, are the automorphisms constructed in the previous subsection. In fact it is enough to prove the above equalities for $f = f_{U, V}$ (if $U \in \calF_V$, $U \neq V$) and $f = g_{U, V}$ (if $U \in \calG_V$).

By the above construction we know that $F' (f_{U, V}) = f_{U, V}$ in the following cases:
\begin{enumerate}

\item
$V = (i, a, b)$ and $U = (i, a, b + 1)$,

\item
$V = (i, a, a + 1 - \delta_{i, 0} \cdot m)$ and $U = (i, a + 1, a + 1 - \delta_{i, 0} \cdot m)$.

\end{enumerate}
We show first that the above formula holds for the remaining arrows of degree $1$, i.e.\ for $V = (i, a, b)$ and $U = (i, a + 1, b)$, where $b > a + 1 - \delta_{i, 0} \cdot m$. Put $W := (i, a, b - 1)$ and $V' := (i, a + 1, b - 1)$. By the above observations and the induction hypothesis
\[
F' (f_{V, W}) = f_{V, W}, \qquad F' (f_{V', W}) = f_{V', W}, \qquad F' (f_{U, V'}) = f_{U, V'}.
\]
Since $F' (f_{U, V}) \in \Hom (U, V)$, $F' (f_{U, V}) = \lambda \cdot f_{U, V} + \mu \cdot g_{U, V}'$, for some $\lambda, \mu \in \Bbbk$. Using that $f_{U, V} \circ f_{V, W} = f_{U, W} = f_{U, V'} \circ f_{V', W}$, we get
\[
F' (f_{U, V}) \circ F' (f_{V, W}) = F' (f_{U, V'}) \circ F' (f_{V', W}).
\]
Consequently,
\[
\lambda \cdot f_{U, W} + \mu \cdot g_{U, V}' \circ f_{V, W} = f_{U, W}.
\]
This immediately implies $\lambda = 1$ and $\mu \cdot g_{U, V}' \circ f_{V, W} = 0$. As a result $F' (f_{U, V}) = f_{U, V}$, provided $g_{U, V}' = 0$, thus assume that $g_{U, V}' \neq 0$. Then $U \in \calG_V$ and this implies $U \in \calG_W$, thus $g_{U, V}' \circ f_{V, W} = g_{U, V} \circ f_{V, W} = g_{U, W} \neq 0$. Consequently, $\mu = 0$ and again $F' (f_{U, V}) = f_{U, V}$.

Now let $V$ and $U$ be arbitrary such that $U \in \calF_V$, $U \neq V$. We know that $f_{U, V}$ is a composition of $\deg f_{U, V}$ arrows of degree $1$ (see the discussion after Corollary~\ref{coro:radical}), hence $F' (f_{U, V}) = f_{U, V}$.

Finally we show that $F' (g_{U, V}) = g_{U, V}$, for all $V \in \Gamma_0$, $U \in \calG_V$. Since $g_{U, V} \in \rad^\infty (V, U)$, while $f_{U, V} \not \in \rad^\infty (V, U)$ (provided $U \in \calF_V$) by Lemma~\eqref{lemm:rad}, there exists $\lambda_{U, V} \in \Bbbk$ such that $F' (g_{U, V}) = \lambda_{U, V} \cdot g_{U, V}$. Our aim is to show that $\lambda_{U, V} = 1$, for all possible $V$ and $U$. The following will be useful.

\begin{lemm} \label{lemm:lambdag}
Fix $i \in [0, n - 1]$ and $a \in \bbZ$. If there exist $b \in [a - \delta_{i, 0} \cdot m, \infty)$ and $x \in (-\infty, a + \delta_{i, n - 1} \cdot m]$, such that $\lambda_{(i + 1, x, a), (i, a, b)} = 1$, then $\lambda_{(i + 1, x', a), (i, a, b')} = 1$, for all $b' \in [a - \delta_{i, 0} \cdot m, \infty)$ and $x' \in (-\infty, a + \delta_{i, n - 1} \cdot m]$.
\end{lemm}

Recall that as usual $i + 1$ is calculated modulo $n$.

\begin{proof}
We first show that $\lambda_{(i + 1, a + \delta_{i, n - 1} \cdot m, a), (i, a, a - \delta_{i, 0} \cdot m)} = 1$. Put $V := (i, a, a - \delta_{i, 0} \cdot m)$, $V' := (i, a, b)$, $U' := (i + 1, x, a)$, and $U := (i + 1, a + \delta_{i, n - 1} \cdot m, a)$. Then $g_{U, V} = f_{U, U'} \circ g_{U', V'} \circ f_{V', V}$, hence
\begin{multline*}
\lambda_{(i + 1, a + \delta_{i, n - 1} \cdot m, a), (i, a, a - \delta_{i, 0} \cdot m)} \cdot g_{U, V} = F' (g_{U, V})
\\
= F' (f_{U, U'}) \circ F' (g_{U', V'}) \circ F' (f_{V', V}) = f_{U, U'} \circ g_{U', V'} \circ f_{V', V} = g_{U, V},
\end{multline*}
thus $\lambda_{(i, a, a - \delta_{i, 0} \cdot m), (i + 1, a + \delta_{i, n - 1} \cdot m, a)} = 1$.

Now let $b'$ and $x'$ be arbitrary, and put $V'' := (i, a, b')$ and $U'' := (i + 1, x', a)$. Similarly as above $g_{U, V} = f_{U, U''} \circ g_{U'', V''} \circ f_{V'', V}$, hence
\begin{multline*}
g_{U, V} = F' (g_{U, V}) = F' (f_{U, U''}) \circ F' (g_{U'', V''}) \circ F' (f_{V'', V})
\\
= f_{U, U''} \circ (\lambda_{(i + 1, x', a), (i, a, b')} \cdot g_{U'', V''}) \circ f_{V'', V} = \lambda_{(i + 1, x', a), (i, a, b')} \cdot g_{U, V},
\end{multline*}
thus $\lambda_{(i + 1, x', a), (i, a, b')} = 1$.
\end{proof}

Fix $i \in [0, n - 1]$. We show that $\lambda_{(i + 1, x, y), (i, a, b)} = 1$, for all $a \in \bbZ$, $b \in [a - \delta_{i, 0} \cdot m, \infty)$, $x \in (-\infty, a + \delta_{i, n - 1} \cdot m]$, $y \in [a, b + \delta_{i, 0} \cdot m]$, in several steps.

$\mathbf{0}^\circ$. $a = 0 = y$.

Put $V := (i, 0, 0)$ and $U := (i + 1, 0, 0)$. By the construction, $\phi_V = \Id_V$ and $\phi_U = \Id_U$. Since $V$ and $U$ are projective modules by Corollary~\ref{coro:projective}, $F (g_{U, V}) = g_{U, V}$. Consequently, $F' (g_{U, V}) = \phi_U \circ F (g_{U, V}) \circ \phi_V^{-1} = g_{U, V}$, i.e.\ $\lambda_{U, V} = 1$. Now the claim in this case follows for arbitrary $b$ and $x$ from Lemma~\ref{lemm:lambdag}.

$\mathbf{1}^\circ$. $a = y$ arbitrary.

Assume first that $a > 0$. Put $V' := (i, 0, a - \delta_{i, 0} \cdot m)$, $V := (i, a, a - \delta_{i, 0} \cdot m)$, $U' := (i + 1, \delta_{i, n - 1} \cdot m, 0)$, and $U := (i + 1, \delta_{i, n - 1} \cdot m, a)$. Then $f_{U, U'} \circ g_{U', V'} = g_{U, V'} = g_{U, V} \circ f_{V, V'}$. By earlier steps
\[
F' (f_{V, V'}) = f_{V, V'}, \qquad F' (f_{U, U'}) = f_{U, U'}, \qquad F' (g_{U', V'}) = g_{U', V'}.
\]
Consequently,
\begin{multline*}
g_{U, V'} = f_{U, U'} \circ g_{U', V'} = F' (f_{U, U'} \circ g_{U', V'})
\\
= F' (g_{U, V} \circ f_{V, V'}) = (\lambda_{U, V} \cdot g_{U, V}) \circ f_{V, V'} = \lambda_{U, V} \cdot g_{U, V'},
\end{multline*}
hence $\lambda_{U, V} = 1$. We use Lemma~\ref{lemm:lambdag} again and the claim follows.

The proof for $a < 0$ is analogous.

$\mathbf{2}^\circ$. $a$ and $y$ are arbitrary.

Put $V := (i, a, b)$, $W := (i, y, b)$, and $U := (i, x, y)$. Then $g_{U, V} = g_{U, W} \circ f_{W, V}$. We already know that $F' (g_{U, W}) = g_{U, W}$ and $F' (f_{W, V}) = f_{W, V}$. Consequently, $F' (g_{U, V}) = g_{U, V}$, i.e.\ $\lambda_{U, V} = 1$.

The following proposition, which in view of Proposition~\ref{prop:standard} constitutes the first important step in the proof of the main result, summarizes the above calculations.

\begin{prop} \label{prop:iso}
There exists a natural isomorphism $\phi \colon F \to \Id_{\bK^b (\proj \Lambda)}$ and a natural isomorphism $\omega' \colon \Sigma \to \Sigma$ such that $\phi$ is a natural isomorphism between triangle functors $(F, \omega)$ and $(\Id_{\bK^b (\proj \Lambda)}, \omega')$.
\end{prop}

\begin{proof}
We know from the preceding discussion that we have a natural isomorphism $\phi \colon F |_{\Bbbk \Gamma / \calI} \to \Id_{\Bbbk \Gamma / \calI}$. The isomorphism $\phi$ extends to a natural isomorphism $\phi \colon F \to \Id_{\bK^b (\proj \Lambda)}$ by~\cite{ChenYe}*{Lemma~2.4}. The existence of $\omega'$ is a consequence of~\cite{ChenYe}*{Lemma~2.3}.
\end{proof}

\subsection{Connecting isomorphism}

Let $(\Id_{\bK^b (\proj \Lambda)}, \omega')$ be a triangle functor. Our aim is to show that the functor $(\Id_{\bK^b (\proj \Lambda)}, \omega')$ is isomorphic to the identify functor (in fact in most of the cases, i.e.\ if $n > 1$ or $m > 0$, already $\omega' = \Id_\Sigma$). We start with the following.

\begin{lemm} \label{lemm:coniso}
Let $(\Id_{\bK^b (\proj \Lambda)}, \omega')$ be a triangle functor. Then there exist scalars $\mu_V \in \Bbbk$, $V \in \Gamma_0$, such that $\omega_V' = f_{\Sigma V, \Sigma V} + \mu_V \cdot g_{\Sigma V, \Sigma V}'$, for all $V \in \Gamma_0$. In particular, $\omega' = \Id_\Sigma$, if $n > 1$.
\end{lemm}

\begin{proof}
Fix $V = (i, a, b) \in \Gamma_0$, and put $U := (i, a, b + 1)$ and $W := (i, b + 1 + \delta_{i, 0} \cdot m, b + 1)$. We show there exists an exact triangle
\begin{equation} \label{eq:trian}
V \xrightarrow{f_{U, V}} U \xrightarrow{f_{W, U}} W \xrightarrow{\nu \cdot g_{\Sigma V, W}} \Sigma V,
\end{equation}
for some $\nu \in \Bbbk$, $\nu \neq 0$. Indeed, one easily verifies the following:
\begin{enumerate}

\item
$f_{W, U} \circ f_{U, V} = 0$,

\item
if $f \circ f_{U, V} = 0$, then there exists $g$ such that $f = g \circ f_{W, U}$,

\item
if $h \circ f_{W, U} = f_{W, U}$, then $h$ is an automorphism.

\end{enumerate}
Consequently, there exists an exact triangle $V \xrightarrow{f_{U, V}} U \xrightarrow{f_{W, U}} W \xrightarrow {g} \Sigma V$, for some $g$, by~\cite{BobinskiSchmude}*{Proposition~2.2}. Observe that $g \neq 0$, since $U$ is not the direct sum of $V$ and $W$ (see~\cite{Happel}*{Lemma~I.1.4}). Consequently, $g = \nu \cdot g_{\Sigma V, W}$, for some $\nu \neq 0$, if either $n > 1$ or $n = 1$ and $\Sigma V \not \in \calF_W$. On the other hand, if $n = 1$ and $\Sigma V \in \calF_W$, then $g = \xi \cdot f_{\Sigma V, W} + \nu \cdot g_{\Sigma V, W}$, for some $\xi, \nu \in \Bbbk$, such that $\xi \neq 0$ or $\nu \neq 0$. Recall that $\Sigma V = (i, a + 1 + m, b + 1 + m)$ in this case. Consequently, $\Sigma V \in \calF_W$ implies $a = b$, which in turn gives $\Sigma V \in \calF_U \setminus \calG_U$. Thus $0 = g \circ f_{W, U} = \xi \cdot f_{\Sigma V, U}$, where the first equality is a consequence of~\cite{Happel}*{Proposition~I.1.2(a)}. This implies $\xi = 0$, i.e.\ $g = \nu \cdot g_{\Sigma V, W}$, for some $\nu \neq 0$.

By applying the functor $(\Id_{\bK^b (\proj \Lambda)}, \omega')$ to the triangle~\eqref{eq:trian} and using~\cite{Happel}*{Axiom~(TR3)}, we get the following commutative diagram:
\[
\vcenter{\xymatrix@C=4\baselineskip{%
V \ar[r]^{f_{U, V}} \ar@{=}[d] & U \ar[r]^{f_{W, U}} \ar@{=}[d] & W \ar[r]^{\nu \cdot g_{\Sigma V, W}} \ar[d]^h & \Sigma V \ar@{=}[d]
\\
V \ar[r]^{f_{U, V}} & U \ar[r]^{f_{W, U}} & W \ar[r]^{\omega_V' \circ (\nu \cdot g_{\Sigma V, W})} & \Sigma V
}},
\]
for some $h$. There exist $\lambda_V, \mu_V, \lambda', \nu' \in \Bbbk$ such that $\omega_V = \lambda_V \cdot f_{\Sigma V, \Sigma V} + \mu_V \cdot g_{\Sigma V, \Sigma V}'$ and $h = \lambda' \cdot f_{W, W} + \mu' \cdot g_{W, W}'$. The commutativity of the middle square means that
\[
f_{W, U} = \lambda' \cdot f_{W, U} + \mu' \cdot g_{W, W}' \circ f_{W, U},
\]
hence $\lambda' = 1$. Similarly, from the commutativity of the rightmost square we get $\nu \cdot g_{\Sigma_V, W} = \lambda_V \cdot \nu \cdot g_{V, W}$, hence $\lambda_V = 1$, which finishes the proof.
\end{proof}

The above lemma settles the case $n > 1$, thus for the rest of the subsection we assume that $n = 1$. In order to simply notation, we write $(a, b)$ instead of $(0, a, b)$, for $(0, a, b) \in \Gamma_0$. The next step is the following.

\begin{lemm} \label{lemm:connect}
Assume $n = 1$ and $m > 0$. If $(\Id_{\bK^b (\proj \Lambda)}, \omega')$ is a triangle functor, then $\omega' = \Id_\Sigma$.
\end{lemm}

\begin{proof}
Fix $a \in \bbZ$. By induction on $b$ we show $\omega_V' = \Id_{\Sigma V}$, for each $V = (a, b) \in \Gamma_0$. If $b = a - m$, then $\Sigma V \not \in \calG_{\Sigma V}$ (since $m > 0$), hence $\omega_V' = f_{\Sigma V, \Sigma V} = \Id_{\Sigma V}$ by Lemma~\ref{lemm:coniso}.

Now assume that $b > a - m$ and put $U := (a, b - 1)$. Lemma~\ref{lemm:coniso} implies $\omega_V' = f_{\Sigma V, \Sigma V} + \mu \cdot g_{\Sigma V, \Sigma V}'$, for some $\mu \in \Bbbk$. Moreover, $\omega_U' = \Id_{\Sigma U}$ by induction hypothesis. Put $f := \Sigma^{-1} f_{\Sigma V, \Sigma U}$. Then
\[
f_{\Sigma V, \Sigma U} = \Sigma f \circ \omega_U' = \omega_V' \circ \Sigma f = f_{\Sigma V, \Sigma U} + \mu \cdot g_{\Sigma V, \Sigma V}' \circ f_{\Sigma V, \Sigma U},
\]
hence $\mu \cdot g_{\Sigma V, \Sigma V}' \circ f_{\Sigma V, \Sigma U} = 0$.  If $g_{\Sigma V, \Sigma V}' \neq 0$, i.e.\ $\Sigma V \in \calG_{\Sigma V}$, then $\Sigma V \in \calG_{\Sigma U}$ (since $m > 0$), thus $g_{\Sigma V, \Sigma V}' \circ f_{\Sigma V, \Sigma U} = g_{\Sigma V, \Sigma U} \neq 0$. Consequently, $\mu = 0$ and the claim follows.
\end{proof}

It remains to consider the case $n = 1$ and $m = 0$. This case has already been treated in~\cite{ChenYe}*{subsection~7.1}. For completeness we include a proof, which slightly differently formulates the corresponding idea from the proof of~\cite{ChenYe}*{Theorem~7.1}.

We need the following observation.

\begin{lemm} \label{lemm:omega}
Assume $n = 1$ and $m = 0$. Let $(\Id_{\bK^b (\proj \Lambda)}, \omega')$ be a triangle functor, $V, U \in \Gamma_0$, and $f \colon V \to U$.
\begin{enumerate}

\item \label{point:omega1}
We have $\omega_{\Sigma^{-1} U}' \circ f = f \circ \omega_{\Sigma^{-1} V}'$.

\item \label{point:omega2}
If $f$ is not an isomorphism, then $\omega_{\Sigma^{-1} U}' \circ f = f = f \circ \omega_{\Sigma^{-1} V}'$.

\end{enumerate}
\end{lemm}

\begin{proof}
The first part follows from the following sequence of equalities using that $\omega'$ is a natural transformation of $\Sigma$:
\[
\omega_{\Sigma^{-1} U}' \circ f = \omega_{\Sigma^{-1} U}' \circ \Sigma (\Sigma^{-1} f) = \Sigma (\Sigma^{-1} f) \circ \omega_{\Sigma^{-1} V}' = f \circ \omega_{\Sigma^{-1} V}'. \]

In order to prove the second part, we use that $\omega_{\Sigma^{-1} V}' = f_{V, V} + \mu' \cdot g_{V, V}'$ and $\omega_{\Sigma^{-1} U}' = f_{U, U} + \mu'' \cdot g_{U, U}'$, for some $\mu', \mu'' \in \Bbbk$, by Lemma~\ref{lemm:coniso}. If $U \not \in \calF_V \cup \calG_V$, then $f = 0$ and the claim is obvious. If $U \in \calG_V \setminus \calF_V$, then $f = \mu \cdot g_{U, V}$, for some $\mu \in \Bbbk$, and the equalities $\omega_{\Sigma^{-1} U}' \circ f = f = f \circ \omega_{\Sigma^{-1} V}'$ follow by easy calculations. Thus we may assume that $U \in \calF_V$. Then there exist $\lambda, \mu \in \Bbbk$ such that $f = \lambda \cdot f_{U, V} + \mu \cdot g_{U, V}'$. Consequently
\[
\omega_{\Sigma^{-1} U}' \circ f = f + (\lambda \cdot \mu') \cdot g_{U, U}' \circ f_{U, V}.
\]
Since $f$ is not an isomorphism, either $\lambda = 0$ or $U \neq V$ by Lemma~\ref{lemm:auto}. In the former case, the equality $\omega_{\Sigma^{-1} U}' \circ f = f$ immediately follows. In the latter case, $U \not \in \calG_V$ (since $n = 1$ and $m = 0$ implies that $V$ is the unique element of $\calF_V \cap \calG_V$), hence $g_{U, U}' \circ f_{U, V} = 0$, and the equality $\omega_{\Sigma^{-1} U}' \circ f = f$ follows again. The equality $f \circ {\Sigma^{-1} V}' = f$ is proved analogously.
\end{proof}

We finish this subsection with following.

\begin{lemm} \label{lemm:connectbis}
Assume $n = 1$ and $m = 0$. If $(\Id_{\bK^b (\proj \Lambda)}, \omega')$ is a triangle functor, then there exists an isomorphism $\eta \colon (\Id_{\bK^b (\proj \Lambda)}, \omega') \to (\Id_{\bK^b (\proj \Lambda)}, \Id_\Sigma)$ of triangle functors.
\end{lemm}

\begin{proof}
We first define automorphisms $\eta_V$, for $V \in \Gamma_0$. Let $V = (a, b)$. If $a = 0$, then we put $\eta_V := \Id_V$. If $a > 0$, then we define $\eta_V$ inductively
by $\eta_V := \Sigma \eta_{\Sigma^{-1} V} \circ \omega_{\Sigma^{-1} V}'$ (note that $\Sigma^{-1} V = (a - 1, b - 1)$).  Similarly, we put $\eta_V := \Sigma^{-1} \eta_{\Sigma V} \circ \Sigma^{-1} \omega_V'^{-1}$, if $a < 0$. One easily checks that the formula
\begin{equation} \label{eq:eta}
\eta_V = \Sigma \eta_{\Sigma^{-1} V} \circ \omega_{\Sigma^{-1} V}'
\end{equation}
holds for arbitrary $V$.

We show that $\eta_U \circ f = f \circ \eta_V$ for arbitrary $V = (a, b), U = (x, y) \in \Gamma_0$ and $f \colon V \to U$. Assume first that $a = x$. If $a = 0$, then the claim is obvious. If $a > 0$, then using Lemma~\ref{lemm:omega}\eqref{point:omega1} and the induction hypothesis, we get
\begin{multline*}
\eta_U \circ f = \Sigma \eta_{\Sigma^{-1} U} \circ \omega_{\Sigma^{-1} U}' \circ f = \Sigma \eta_{\Sigma^{-1} U} \circ f \circ \omega_{\Sigma^{-1} V}' = \Sigma (\eta_{\Sigma^{-1} U} \circ \Sigma^{-1} f) \circ \omega_{\Sigma^{-1} V}'
\\
= \Sigma (\Sigma^{-1} f \circ \eta_{\Sigma^{-1} V}) \circ \omega_{\Sigma^{-1} V}' = f \circ \Sigma \eta_{\Sigma^{-1} V} \circ \omega_{\Sigma^{-1} V}' = f \circ \eta_V.
\end{multline*}
Similarly, the claim follows if $a < 0$. On the other hand, if $a \neq x$, then $V \neq U$, in particular $f$ is not an isomorphism. Similarly as above we show by induction on $|a|$ ($|x|$) and using Lemma~\ref{lemm:omega}\eqref{point:omega2} that $f \circ \eta_V = f$ ($\eta_U \circ f = f$, respectively).

In other words, we have just showed that $\eta$ is a natural automorphism of $\Id |_{\Bbbk \Gamma / \calI}$. By~\cite{ChenYe}*{Lemma~2.4} $\eta$ extends to a natural automorphism of $\Id_{\bK^b (\proj \Lambda)}$. Using \cite{ChenYe}*{Lemma~2.3} we obtain a natural automorphism $\omega''$ of $\Sigma$ such that $\eta \colon (\Id_{\bK^b (\proj \Lambda)}, \omega') \to (\Id_{\bK^b (\proj \Lambda)}, \omega'')$ is an isomorphism of triangle functors. Observe that
\[
\omega_V'' = \Sigma \eta_V \circ \omega_V' \circ \eta_{\Sigma V}^{-1} = \Sigma \eta_V \circ \omega_V' \circ \omega_{\Sigma^{-1} \Sigma V}'^{-1} \circ \Sigma \eta_{\Sigma^{-1} \Sigma V}^{-1} = \Id_{\Sigma V},
\]
for every $V \in \Gamma_0$, where the second equality follows from~\eqref{eq:eta}. Consequently, $\omega'' = \Id_\Sigma$ and the claim follows.
\end{proof}

\subsection{Summary}

We summarize now the above considerations.

\begin{prop} \label{prop:summary}
Let $\Lambda := \Lambda (n, m)$, for $n \in \bbN_+$ and $m \in \bbN$.
\begin{enumerate}

\item \label{point:summary1}
If $(F, \omega)$ is a pseudo-identify on $\bK^b (\proj \Lambda)$, then $(F, \omega)$ is isomorphic, as a triangle functor, to $(\Id_{\mathbf{K}^b(\proj \Lambda)}, \Id_\Sigma)$.

\item \label{point:summary2}
In particular, if $G \colon \bD^b (\mod \Lambda) \to \bD^b (\mod B)$ is a triangle equivalence, for some algebra $B$, then $G$ is standard.

\end{enumerate}
\end{prop}

\begin{proof}
\eqref{point:summary1}~It follows from Propostion~\ref{prop:iso} that there exists a natural isomorphism $\phi \colon F \to \Id_{\bK^b (\proj \Lambda)}$ and a natural isomorphism $\omega' \colon \Sigma \to \Sigma$ such that $\phi$ is a natural isomorphism between triangle functors $(F, \omega)$ and $(\Id_{\bK^b (\proj \Lambda)}, \omega')$. Now Lemmas~\ref{lemm:coniso}, \ref{lemm:connect} and~\ref{lemm:connectbis} imply that there exists a natural isomorphism $\eta \colon (\Id_{\bK^b (\proj \Lambda)}, \omega') \to (\Id_{\bK^b (\proj \Lambda)}, \Id_\Sigma)$ of triangle functors. By taking $\eta \circ \phi$ we get our claim.

\eqref{point:summary2}~This follows immediately from~\eqref{point:summary1} and Proposition~\ref{prop:standard}.
\end{proof}

\subsection{Proof of Theorem~\ref{theo:main_special}}

Let $F \colon \bD^b (\mod A) \to \bD^b (\mod B)$ be a derived equivalence between derived discrete algebras $A$ and $B$ of infinite global dimension. By Proposition~\ref{prop:discrete} there exist $n \in \bbN_+$ and $m \in \bbN$ such that $A$ is derived equivalent to $\Lambda := \Lambda (n, m)$. By \cite{Rickard}*{Theorem~3.3} we know there exists a standard derived equivalence $H \colon \bD^b (\mod A) \to \bD^b (\mod \Lambda)$. Moreover, there exists an equivalence $G \colon \bD^b (\mod \Lambda) \to \bD^b (\mod B)$ such that the functors $F$ and $G \circ H$ are isomorphic. Now $G$ is standard by Proposition~\ref{prop:summary}\eqref{point:summary2}. Since a composition of standard derived equivalences is obviously standard, the claim follows. \qed

\subsection*{Funding}

The both authors gratefully acknowledge the support of the National Science Centre grant no.~2020/37/B/ST1/00127.

\bibsection

\end{document}